\documentclass[12pt,a4paper]{article}
\usepackage{aligned-overset}
\usepackage{amsmath,amsthm}
\usepackage{authblk}
\usepackage[style=numeric-comp,maxbibnames=99,bibencoding=utf8,firstinits=true,backend=bibtex]{biblatex}
\usepackage[hmargin={26mm,26mm},vmargin={30mm,35mm}]{geometry}
\usepackage{graphicx}
\usepackage{hyperref}
\usepackage{newtxtext,newtxmath}
\usepackage{nicefrac}
\usepackage{xcolor}
\usepackage{caption}
\usepackage{subcaption}

\bibliography{ddrin}

\graphicspath{{./figures/}}

\newcommand{\email}[1]{\href{mailto:#1}{#1}}

\theoremstyle{plain}
\newtheorem{theorem}{Theorem}
\newtheorem{proposition}[theorem]{Proposition}
\newtheorem{lemma}[theorem]{Lemma}

\theoremstyle{remark}
\newtheorem{remark}[theorem]{Remark}

\theoremstyle{definition}

\newcommand{\st}{\;:\;}
\newcommand{\Real}{\mathbb{R}}

\newcommand{\INT}{\mathrm{int}}
\newcommand{\EXT}{\mathrm{ext}}

\newcommand{\Th}{\mathcal{T}_h}
\newcommand{\Eh}{\mathcal{E}_h}
\newcommand{\Vh}{\mathcal{V}_h}
\newcommand{\VGh}{\mathcal{V}_{\Gamma,h}}

\newcommand{\ET}{\mathcal{E}_T}
\newcommand{\TE}{\mathcal{T}_E}

\newcommand{\EGh}{\mathcal{E}_{\Gamma,h}}

\newcommand{\Poly}[1]{\mathcal{P}^{#1}}
\newcommand{\cRoly}[1]{\mathcal{R}^{{\rm c},#1}}

\newcommand{\lproj}[2]{\pi^{#1}_{#2}}

\newcommand{\GT}{G_T^k}
\newcommand{\Gh}{G_h^k}
\newcommand{\pT}{p_T^{k+1}}

\newcommand{\Ih}{\underline{I}_h^k}
\newcommand{\IT}{\underline{I}_T^k}

\newcommand{\jump}[1]{[#1]_E}
\newcommand{\avg}[1]{\{#1\}_{k,E}}
\newcommand{\savg}[1]{\{#1\}_{\overline{k},E}}

\newcommand{\jumpG}[1]{[#1]_\Gamma}

\newcommand{\tF}{t_{\rm f}}

\title{A discrete de Rham discretization of interface diffusion problems with application to the Leaky Dielectric Model}
\author{Daniele A. Di Pietro}
\author{Simon Mendez}
\author{Aurelio E. Spadotto}
\affil{IMAG, Univ. Montpellier, CNRS, Montpellier, France,  
  \email{daniele.di-pietro@umontpellier.fr},
  \email{simon.mendez@umontpellier.fr},
  \email{aurelio-edoardo.spadotto@umontpellier.fr}
}

\begin{document}

\maketitle

\begin{abstract}
  Motivated by the study of the electrodynamics of particles, we propose in this work an arbitrary-order discrete de Rham scheme for the treatment of elliptic problems with potential and flux jumps across a fixed interface.
  The scheme seamlessly supports general elements resulting from the cutting of a background mesh along the interface.
  Interface conditions are enforced weakly \emph{à la Nitsche}.
  We provide a rigorous convergence of analysis of the scheme for a steady model problem and showcase an application to a physical problem inspired by the Leaky Dielectric Model. 
\end{abstract}



\section{Introduction}

The study of the dynamics of particles in an electric field is motivated by  numerous applications. Subjecting cells to an electric field is indeed a non-invasive and label-free method to obtain a high-throughput characterization of individual cells \cite{Sun.Morgan:10}. A classical device is the Coulter counter \cite{Coulter:53}, developed in the 1950s and still used in blood analyzers to count and size red blood cells \cite{Taraconat.Gineys.ea:21} using a DC field. Nowadays, the use of AC fields combined with the versatility of microfluidics make 
single-cell microfluidic impedance cytometry a promising technique for multiparametric cell characterization  \cite{Honrado.Bisegna.ea:21}.  
Other processes in which cells are subjected to an electric field include dielectrophoresis, a common technique used to manipulate or deform cells \cite{Du.Dao.ea:14}, and electroporation, in which pores are created in the membrane to enable exchange of molecules between the internal and the suspending media \cite{Neumann.Sowers.ea:89}.
From a more fundamental point of view, subjecting simpler particles like drops or vesicles to electrical fields has revealed fascinating dynamics when the internal and the external fluids have different conductivity and permittivity \cite{Vlahovska:19}: symmetry-breaking (rotation of the particle), large transient deformations with sharp edges for vesicles\ldots  

Motivated by the problem of the deformation of a drop in a steady electric field  \cite{Taylor:66}, G. I. Taylor introduced the so-called Leaky Dielectric Model, in which bulk fluids are assumed to be free of charge, so that the coupling between between electric and mechanical phenomena occurs at the interface. Since then, the Leaky Dielectric Model has been shown to provide valuable theoretical insights in the field of electrohydrodynamics of particles \cite{Saville:97,Vlahovska:19}.
In the context of this model, the electrostatic potential is a scalar-valued field which is harmonic in the bulk while, on the droplet interface, fulfills conditions that express the continuity of current through the membrane as well as a drop in potential called \emph{Galvani potential} \cite{Mori.Young:18}.
When seeking numerical solutions to this problem, tracking the position of the interface by means of an adapted mesh can be challenging and impact on the cost of the simulation.
This calls for solutions where the mesh is not redesigned after each displacement of the membrane.

The literature on numerical methods for interface problems is vast, and we will limit ourselves here to works that bear relations with the present approach.
A large number of methods rely on a background mesh which is not compliant with the interface, and are therefore referred to as \emph{unfitted}.
In the Generalized/Extended Finite Element method \cite{Sukumar.Moes.ea:00,Strouboulis.Babuska.ea:00}, non-polynomial functions with compact support are added to the discrete space in order to capture the behavior at the interface.
In the Immersed Finite Element method \cite{Adjerid.Babuska.ea:23}, the added functions are piecewise polynomials.
In the CutFEM method \cite{Burman.Claus.ea:15}, interface conditions are taken into account by using discontinuous basis functions inside the elements cut by the interface and relying on Nitsche's techniques for their enforcement.
The CutFEM principles have been recently extended to the Hybrid High-Order method \cite{Burman.Ern:18,Burman.Ern:19}, which supports much more general element geometries than standard Finite Elements; see \cite{Di-Pietro.Ern.ea:14,Di-Pietro.Ern:15,Di-Pietro.Droniou:20}.

In the present work, we consider a method based on meshes obtained by cutting elements crossed by the interface.
No special treatment (such as the addition of ad hoc functions) is required for discretization of the diffusion operator on the elements resulting from the cut.
This is because the discretization in the bulk hinges on the $H^1$-like space of \cite{Di-Pietro.Droniou:23}, which supports general polygonal or polyhedral meshes.
As a result, our method enjoys both the assets of unfitted methods (since cuts can occur at arbitrary locations, making the approximation of the interface totally independent of the background mesh) and of fitted methods (since interface elements do not require a special treatment).
The support of general polytopal elements additionally provides an effective means to counter the possible degradation of mesh quality resulting from the cuts, since pathological elements can be merged into neighbors in the spirit of \cite{Bassi.Botti.ea:12,Antonietti.Giani.ea:13,Johansson.Larson:13}; see also \cite{Huang.Wu.ea:17} for an application of these ideas to conforming finite elements.
Interface conditions are enforced weakly through subtly defined terms that ensure consistency; see Remark~\ref{rem:comparison.dG} for further insight into this point.
The design of such terms, based on the use of trace reconstructions, is indeed one of the main contributions of the present work.
Robustness with respect to the jumps of the diffusion coefficient is obtained by using diffusion-dependent weighted averages in the spirit of \cite{Burman.Zunino:06,Di-Pietro.Ern.ea:08}.
To keep the exposition as simple as possible, we focus here on the two-dimensional case with interfaces approximated by closed polygonal chains and consider numerically the case of curved interfaces.

The rest of this work is organized as follows.
In Section~\ref{sec:continuous.setting} we state the continuous problem.
The discrete setting (mesh, polynomial spaces) is introduced in Section~\ref{sec:discrete.setting}.
In Section~\ref{sec:discrete.problem} we describe the construction leading to the discrete problem and state the scheme as well as the main results of the analysis.
A comprehensive set of numerical tests is carried out in Section~\ref{sec:numerical.tests}, while the application to the Leaky Dielectric Model makes the object of Section~\ref{sec:ldm}.
Finally, the proofs of the stability and error estimates results stated in Section~\ref{sec:discrete.problem:main.results} are provided in Section~\ref{sec:proofs}.


\section{Continuous setting}\label{sec:continuous.setting}

To keep the description of the method as simple as possible, we focus on the two-dimensional case.
We emphasize, however, that it is possible to extend the present method to three space dimensions in the spirit of \cite{Di-Pietro.Droniou:23}, as well as to curved approximations of the interface by adapting the techniques of \cite{Botti.Di-Pietro:18,Beirao-da-Veiga.Russo.ea:19,Yemm:24}.

\begin{figure}
  \centering
  \includegraphics[width=0.4\textwidth]{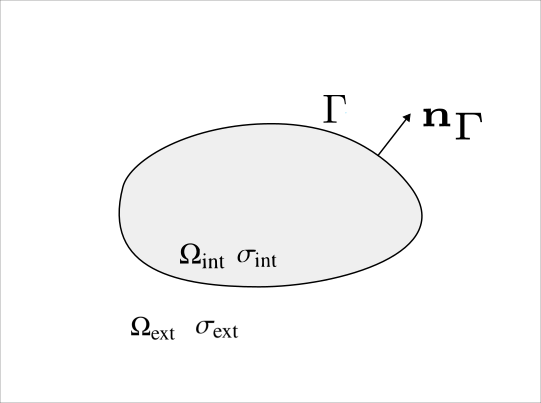}
  \caption{Configuration for the continuous problem.}\label{fig:basic_sketch}
\end{figure}

Consider an open bounded connected polygonal domain $\Omega\subset\Real^2$ with boundary $\partial \Omega$.
Let $\Gamma \subset \Omega$ be a closed non-intersecting polygonal chain such that $\Gamma \cap \partial \Omega = \emptyset$ (see Figure \ref{fig:basic_sketch}).
The domain is partitioned by $\Gamma$ into an internal and an external polygonal subdomains, respectively denoted by $\Omega_\INT$ and $\Omega_\EXT$ in  what follows.
We denote by $n$ the unit vector field normal to $\Gamma$ and pointing out of $\Omega_\INT$.

Given a couple of functions $v = (v_\INT, v_\EXT)$ with $v_\bullet : \Omega_\bullet \to \Real$ for $\bullet \in \{ \INT, \EXT \}$, each smooth enough to admit a trace on $\Gamma$, we define the following interface jump operator:
\begin{equation}\label{eq:jumpG}
  \jumpG{v} \coloneq v_\INT - v_\EXT.
\end{equation}
When applied to couples of vector-valued functions, the jump operator acts component-wise.
In what follows, whenever needed, we tacitly identify $v$ with the function $v_\INT I_{\Omega_\INT} + v_\EXT I_{\Omega_\EXT} : \Omega \setminus \Gamma \to \Real$, where $I_{\Omega_\bullet}$ is the characteristic function of $\Omega_\bullet$.

Consider a region-wise constant diffusion coefficient $\sigma : \Omega \setminus \Gamma \rightarrow\Real$ such that $\sigma_{|\Omega_\INT} \equiv \sigma_\INT > 0$ and  $\sigma_{|\Omega_\EXT} \equiv \sigma_\EXT > 0$.
Let $f:\Omega \setminus \Gamma \to \Real$, $\Phi_\Gamma: \Gamma \rightarrow \Real^2$, and $J_\Gamma: \Gamma\rightarrow \Real$ be given functions, which we assume smooth enough for the following discussion to make sense.
We consider the problem of finding the couple of scalar-valued functions $u = (u_\INT, u_\EXT)$ with $u_\bullet : \Omega_\bullet \to \Real$, $\bullet \in \{ \INT, \EXT \}$, such that
\begin{subequations}\label{eq:strong}
  \begin{alignat}{4}
    -\nabla\cdot(\sigma_\bullet \nabla u_\bullet) &= f
    &\qquad& \text{in $\Omega_\bullet$, $\bullet \in \{ \INT, \EXT \}$}, \label{eq:strong:diff_eq}
    \\
    \jumpG{u} &= J_\Gamma &\qquad& \text{on $\Gamma$}, \label{eq:strong:jump_phi}
    \\
    \jumpG{\sigma\nabla u} \cdot n_\Gamma &= \Phi_\Gamma &\qquad& \text{on $\Gamma$}, \label{eq:strong:jump_j}      
    \\
    u_\EXT &= 0 &\qquad& \text{on $\partial\Omega$}. \label{eq:strong:bnd_cond}
  \end{alignat}
\end{subequations}


\section{Discrete setting}\label{sec:discrete.setting}

\subsection{Mesh}

We discretize the domain with a polygonal mesh $(\Th, \Eh)$ in the sense of \cite[Definition~1.4]{Di-Pietro.Droniou:20}, with $\Th$ collecting the mesh elements and $\Eh$ the mesh edges.
We additionally denote by $\Vh$ the set of vertices collecting the edge endpoints.
The mesh is assumed to be fitted to the interface, i.e., there exists a subset $\EGh$ of $\Eh$ such that $\Gamma = \bigcup_{E \in \EGh} \overline{E}$.

\begin{remark}[Fitted mesh]
  It is important to emphasize that fitted polytopal meshes supported by the present method can simply be obtained cutting the elements of a background mesh along the interface, as in the numerical tests of Sections~\ref{sec:numerical.tests} and~\ref{sec:ldm}.
  This is a major advantage, particularly in the framework of moving interface problems, as it avoids a potentially expensive remeshing step.
  We also notice that the degradation of mesh quality can also be countered leveraging the support of polytopal elements: whenever an elongated or distorted element results from the cutting, it can be agglomerated into neighboring elements in the spirit of \cite{Bassi.Botti.ea:12,Antonietti.Giani.ea:13,Johansson.Larson:13} to restore mesh quality.
\end{remark}

\begin{remark}[Approximation of the interface]
  The approximation of the interface is totally independent of the background mesh.
  This is important when considering the more general case of curved interfaces, where a finer discretization can be needed.
  Notice that this can lead to polygonal elements with asymptotically small edges.
  In the spirit of \cite{Droniou.Yemm:22}, it can be shown that this kind of elements can be supported by the present method; see also \cite{Cangiani.Dong.ea:17} on this subject.
\end{remark}

For $\bullet \in \{ \INT, \EXT \}$, we let $\Th^\bullet \coloneq \left\{ T \in \Th \st T \subset \Omega_\bullet \right\}$.
For any element $T \in \Th$, we denote by $\ET$ the set collecting its edges.
Symmetrically, the set of elements sharing one edge $E \in \Eh$ is denoted by $\TE$.
For each edge $E \in \Eh$, we fix once and for all a unit normal vector $n_E$ and, for any $T \in \TE$, we denote by $\omega_{TE} \in \{ -1, 1 \}$ the relative orientation of $E$ with respect to $T$, such that $\omega_{TE} n_E$ points out of $T$.

In what follows, we assume that the mesh we are working on belongs to a regular sequence in the sense of \cite[Definition~1.9]{Di-Pietro.Droniou:20}.
Given a mesh element or edge $X \in \Th \cup \Eh$, we denote by $h_X$ its diameter, so that $h = \max_{T \in \Th} h_T$.
We will abbreviate by $a \lesssim b$ the inequality $a \le C b$ with $C$ independent of $h$, $\sigma$ and, for local inequalities, the corresponding element or edge.
Further details on the dependence of the hidden constants will be provided when appropriate.

\subsection{Local polynomial spaces}

Given $X \in \Th \cup \Eh$ and an integer $m \ge 0$, we denote by $\Poly{m}(X)$ the space spanned by the restriction of polynomials of the space variables to $X$ and by $\lproj{m}{X}$ the corresponding $L^2$-orthogonal projector.
We conventionally set $\Poly{-1}(X) \coloneq \{ 0 \}$.
For all $T \in \Th$, we will also need the space
\[
\cRoly{m}(T) \coloneq (x - x_T) \Poly{m-1}(T),
\]
where $x_T$ is a point inside $T$ at a distance from its boundary comparable to $h_T$.
It can be proved that the divergence from $\cRoly{m}(T)$ to $\Poly{m-1}(T)$ is an isomorphism; cf. \cite[Corollary~7.3]{Arnold:18}.


\section{Discrete problem}\label{sec:discrete.problem}

\subsection{Discrete space}

For any $k \ge 0$ and $\bullet \in \{ \INT, \EXT \}$, we let
\[
\underline{V}_{\bullet,h}^k \coloneq
\begin{aligned}[t]
  \Big\{
  \underline{v}_h = \big( (v_T)_{T \in \Th^\bullet}, (v_E)_{E \in \Eh^\bullet}, (v_V)_{V \in \Vh^\bullet} \big) \st
  &\text{$v_T \in \Poly{k-1}(T)$ for all $T \in \Th^\bullet$,}
  \\
  &\text{$v_E \in \Poly{k-1}(E)$ for all $E \in \Eh^\bullet$,}
  \\
  &\text{$v_V \in \Real$ for all $V \in \Vh^\bullet$}
  \Big\}.
\end{aligned}
\]
We consider the discrete space
\[
\underline{V}_h^k \coloneq \underline{V}_{\INT,h}^k \times \underline{V}_{\EXT,h}^k.
\]
as well as its subspace $\underline{V}_{h,0}^k$ with edge and vertex values vanishing on $\partial \Omega$.

The interpolator $\Ih : C^0(\overline{\Omega}_\INT) \times C^0(\overline{\Omega}_\EXT) \mapsto \underline{V}_h^k$ is such that, for all $v = (v_\INT, v_\EXT) \in C^0(\overline{\Omega}_\INT) \times C^0(\overline{\Omega}_\EXT)$,
\[
\Ih v \coloneq \big( \underline{I}_{\INT,h}^k v_\INT, \underline{I}_{\EXT,h}^k v_\EXT \big),
\]
where, for $\bullet \in \{ \INT, \EXT \}$,
\[
\underline{I}_{\bullet,h}^k v_\bullet \coloneq \big(
(\lproj{k-1}{T} v_\bullet)_{T \in \Th^\bullet},
(\lproj{k-1}{E} v_\bullet)_{E \in \Eh^\bullet},
(v_\bullet(x_V))_{V \in \Vh^\bullet}
\big).
\]

For all $T \in \Th$, we respectively denote the restrictions of $\underline{V}_h^k$, $\underline{v}_h \in \underline{V}_h^k$, and $\Ih$ to $T$ by $\underline{V}_T^k$, $\underline{v}_T \in \underline{V}_T^k$, and $\IT$.
Such restrictions are obtained collecting the polynomial components on $T$ and its boundary.
Since every mesh element is contained in one and only one subdomain, there is no ambiguity on which vertex and edge components to select when restricting to an element $T$ such that $\partial T \cap \Gamma \neq \emptyset$.

\subsection{Element gradient and potential}

For any $T \in \Th$, any $\underline{v}_T \in \underline{V}_T^k$, and any $E \in \ET$, we let the edge potential $v_{TE}$ be the unique function in $\Poly{k+1}(E)$ such that
\[
\text{%
  $v_{TE}(x_V) = v_V$ for any endpoint $V$ of $E$ and $\lproj{k-1}{E} v_{TE} = v_E$.
}
\]

\begin{remark}[Edge potential]
  Clearly, $v_{TE} \equiv 0$ whenever $E \subset \partial \Omega$ is a boundary edge and $\underline{v}_h \in \underline{V}_{h,0}^k$.
  On the other hand, when $E \subset (\partial T_1 \cap \partial T_2) \setminus \Gamma$ is an internal edge that does not lie on the interface, the value of the edge potential does not depend on the element, i.e., $v_{T_1E} = v_{T_2E}$.
\end{remark}

We define the discrete gradient $\GT : \underline{V}_T^k \to \Poly{k}(T)^2$ and potential $\pT : \underline{V}_T^k \to \Poly{k+1}(T)$ such that, for all $\underline{v}_T \in \underline{V}_T^k$,
\begin{gather} \nonumber
  \int_T \GT \underline{v}_T \cdot \tau
  = - \int_T v_T (\nabla \cdot \tau)
  + \sum_{E \in \ET} \omega_{TE} \int_E v_{TE} (\tau \cdot n_E)
  \qquad \forall \tau \in \Poly{k}(T)^2,
  \\ \label{eq:pT}
  \int_T \pT \underline{v}_T (\nabla \cdot \tau)
  = - \int_T \GT \underline{v}_T \cdot \tau
  + \sum_{E \in \ET} \omega_{TE} \int_E v_{TE} (\tau \cdot n_E)
  \qquad \forall \tau \in \cRoly{k+2}(T).
\end{gather}

\begin{remark}[Validity of \eqref{eq:pT}]\label{eq:validity.pT}
  Following \cite[Remark~17]{Di-Pietro.Droniou:23}, the relation \eqref{eq:pT} actually holds for all $\tau \in \Poly{k}(T)^2 + \cRoly{k+2}(T)$.
\end{remark}

Accounting for the previous remark, we notice that, integrating by parts the left-hand side of \eqref{eq:pT} and rearranging, we have, for all $(\underline{v}_T, \tau) \in \underline{V}_T^k \times (\Poly{k}(T)^2 + \cRoly{k+2}(T))$,
\begin{equation}\label{eq:pT.bis}
  \int_T \nabla \pT \underline{v}_T \cdot \tau
  = \int_T \GT \underline{v}_T \cdot \tau
  + \sum_{E \in \ET} \omega_{TE} \int_E (\pT \underline{v}_T - v_{TE}) (\tau \cdot n_E).
\end{equation}
Selecting $\tau = \nabla \pT \underline{v}_T$ (this is possible since $\nabla \pT \underline{v}_T \in \Poly{k}(T)^d$) in the above expression, using Cauchy--Schwarz, $(2,\infty,2)$-H\"older, and trace inequalities along with $\| n_E \|_{L^\infty(E)^2} \le 1$ in the right-hand side, and simplifying, we get
\begin{equation}\label{eq:est.grad.pT}
  \| \nabla \pT \underline{v}_T \|_{L^2(T)^2}
  \lesssim \left(
  \| \GT \underline{v}_T \|_{L^2(T)^2}^2
  + h_T^{-1} \sum_{E \in \ET} \| \pT \underline{v}_T - v_{TE} \|_{L^2(E)}^2
  \right)^{\nicefrac12}.
\end{equation}

Let $v \in H^{r+2}(T)$ for some $r \in \{0,\ldots,k\}$ and set $\widehat{\underline{v}}_T \coloneq \IT v$.
Using the techniques of \cite{Di-Pietro.Droniou:23}, where the three-dimensional case is considered, it can be proved that
\begin{gather}\label{eq:vTE:approximation}
  \| \widehat{v}_{TE} - v \|_{L^2(E)} \lesssim h_T^{r+\nicefrac32} | v |_{H^{r+2}(T)} \qquad \forall E \in \ET,
  \\ \label{eq:GT:approximation}
  \| \GT \widehat{\underline{v}}_T - \nabla v \|_{L^2(T)^2}
  + h_T^{\nicefrac12} \| \GT \widehat{\underline{v}}_T - \nabla v \|_{L^2(\partial T)^2}
  \lesssim h_T^{r+1} | v |_{H^{r+2}(T)},
  \\ \label{eq:pT:approximation}
  \| \pT \widehat{\underline{v}}_T - v \|_{L^2(T)}
  + h_T^{\nicefrac12} \| \pT \widehat{\underline{v}}_T - v \|_{L^2(\partial T)}
  \lesssim h_T^{r+2} | v |_{H^{r+2}(T)}.
\end{gather}

For future use, we also define the global discrete gradient operator $\Gh : \underline{V}_h^k \to \Poly{k}(\Th)^2$ such that, for all $\underline{v}_h \in \underline{V}_h^k$,
\[
(\Gh \underline{v}_h)_{|T} \coloneq \GT \underline{v}_T
\qquad\forall T \in \Th.
\]

\subsection{Interface trace operators}\label{sec:discrete.problem:trace.operators}

Let $E \in \EGh$ and denote by $T_\INT \in \Th^\INT$ and $T_\EXT \in \Th^\EXT$ the unique elements such that $E \subset \partial T_\INT \cap \partial T_\EXT$.
Notice that, while such elements clearly depend on $E$, we do not highlight this dependency in the notation as it will be clear from the context.
We define the edge jump $\jump{\cdot} : \underline{V}_h^k \to \Poly{k+1}(E)$ and skewed average $\savg{\cdot} : \underline{V}_h^k \to \Poly{k+1}(E)$ operators such that, for all $\underline{v}_h \in \underline{V}_h^k$,
\begin{equation}\label{eq:jump.savg}
  \jump{\underline{v}_h} \coloneq v_{T_\INT E} - v_{T_\EXT E},\qquad
  \savg{\underline{v}_h} \coloneq \lambda_\EXT v_{T_\INT E} + \lambda_\INT v_{T_\EXT E},
\end{equation}
with $\lambda_\INT$ and $\lambda_\EXT$ such that
\begin{equation}\label{eq:k.int.ext}
  \lambda_\INT \coloneq \frac{\sigma_\EXT}{\sigma_\INT + \sigma_\EXT},\qquad
  \lambda_\EXT \coloneq \frac{\sigma_\INT}{\sigma_\INT + \sigma_\EXT}.
\end{equation}

We additionally let, for any vector-valued field $\Psi$ smooth enough to admit a possibly two-valued trace on $E$,
\begin{equation}\label{eq:avg}
  \avg{\Psi} \coloneq \lambda_\INT \gamma_{T_\INT E} \Psi + \lambda_\EXT \gamma_{T_\EXT E} \Psi,
\end{equation}
where $\gamma_{T_\bullet E} \Psi$ denotes the trace of $\Psi_{|T_\bullet}$ on $E$.

\subsection{Discrete problem}

Let
\[
\sigma_T \coloneq \sigma_{|T} \qquad \forall T \in \Th.
\]
We define the bilinear form $a_h : \underline{V}_h^k \times \underline{V}_h^k \to \Real$ and the linear form $\ell_h : \underline{V}_h^k \to \Real$ such that, for all $(\underline{w}_h, \underline{v}_h) \in \underline{V}_h^k \times \underline{V}_h^k$,
\begin{equation}\label{eq:ah}
  \begin{aligned}
    a_h(\underline{w}_h, \underline{v}_h)
    &\coloneq
    \sum_{T \in \Th} \left(
    \int_T \sigma_T \GT \underline{w}_T \cdot \GT \underline{v}_T
    + \frac{\sigma_T}{h_T} \sum_{E \in \ET} \int_E (\pT \underline{w}_T - w_{TE}) (\pT \underline{v}_T - v_{TE})
    \right)
    \\
    &\quad
    - \sum_{E \in \EGh} \int_E \avg{\sigma \Gh \underline{w}_h}\cdot n_E \jump{\underline{v}_h}
    + \eta \sum_{E \in \EGh} \frac{\alpha}{h_E} \int_E \jump{\underline{w}_h} \jump{\underline{v}_h},
  \end{aligned}
\end{equation}
where
\begin{equation}\label{eq:alpha}
  \alpha \coloneq \frac{2 \sigma_\INT \sigma_\EXT}{\sigma_\INT + \sigma_\EXT}
\end{equation}
and
\begin{equation}\label{eq:lh}
  \ell_h(\underline{v}_h)
  \coloneq \sum_{T \in \Th} \int_T f \pT \underline{v}_T
  + \sum_{E \in \EGh} \int_E \Phi_\Gamma \savg{\underline{v}_h}  
  + \eta \sum_{E \in \EGh} \frac{\alpha}{h_E} \int_E J_\Gamma \jump{\underline{v}_h}.
\end{equation}
The discrete problem reads:
Find $\underline{u}_h \in \underline{V}_{h,0}^k$ such that
\begin{equation}\label{eq:discrete}
  a_h(\underline{u}_h, \underline{v}_h)
  = \ell_h(\underline{v}_h)
  \qquad \forall \underline{v}_h \in \underline{V}_{h,0}^k.
\end{equation}

\begin{remark}[Formulation of the interface terms]\label{rem:comparison.dG}
  It is worth noticing that the scheme above is not a simple extension of discontinuous Galerkin techniques to the interface problem of Section~\ref{sec:continuous.setting} with the element potential playing the role of the discontinuous solution inside trace operators.
  On the contrary, such operators use the edge potential in a subtle way, which is required for optimal order consistency; see, in particular, the passages leading to \eqref{eq:lh:reformulation} in the proof of Lemma~\ref{lem:consistency} below.
\end{remark}

\subsection{Energy norm and interface jump seminorm}

In order to state the main results of the theoretical analysis of the method, we define on $\underline{V}_h^k$ the energy norm $\| \cdot \|_{{\rm en},h}$ such that, for all $\underline{v}_h \in \underline{V}_h^k$,
\begin{equation}\label{eq:energy.norm}
  \| \underline{v}_h \|_{{\rm en},h}^2
  \coloneq
  \sum_{T \in \Th} \sigma_T \left(
  \| \GT \underline{v}_T \|_{L^2(T)^2}^2
  +  h_T^{-1} \sum_{E \in \ET} \|\pT \underline{v}_T - v_{TE} \|_{L^2(E)}^2
  \right)
  + | \underline{v}_h |_{{\rm J},h}^2,
\end{equation}
where the interface jump seminorm is such that
\begin{equation}\label{eq:seminorm.J}
  | \underline{v}_h |_{{\rm J},h}^2
  \coloneq \sum_{E \in \EGh} \frac{\alpha}{h_E} \| \jump{\underline{v}_h} \|_{L^2(E)}^2.
\end{equation}

\begin{proposition}[Energy norm]
  The map $\| \cdot \|_{{\rm en},h}$ defines a norm on $\underline{V}_{h,0}^k$.
\end{proposition}

\begin{proof}
  $\| \cdot \|_{{\rm en},h}$ is clearly a seminorm, so we only have to prove that, for all $\underline{v}_h \in \underline{V}_{h,0}^k$, $\| \underline{v}_h \|_{{\rm en},h} = 0$ implies $\underline{v}_h = \underline{0}$.
  The condition $\| \underline{v}_h \|_{{\rm en},h} = 0$ implies:
  (i) for all $T \in \Th$, $\GT \underline{v}_T = 0$ and $\pT \underline{v}_T = v_{TE}$ for all $E \in \ET$ and
  (ii) for all $E \in \EGh$, $\jump{\underline{v}_h} = 0$.
  By \eqref{eq:est.grad.pT}, point (i) implies, in turn, that $\pT \underline{v}_T$ is constant on each $T \in \Th$.
  Since $v_{TE} = 0$ whenever $E \in \Eh$ is a boundary edge contained in $\partial \Omega$ and $T \in \Th$ is the unique mesh element to which it belongs, $\pT \underline{v}_T = 0$.
  Proceeding from $\partial \Omega$ towards the interior of $\Omega_\EXT$, this gives $\pT \underline{v}_T = 0$ for all $T \in \Th^\EXT$ and $v_{TE} = 0$ for all $E \in \ET$.
  These two conditions combined show that all element, edge, and vertex components of $\underline{v}_h$ in $\Omega_\EXT$ vanish.
  Since interface jumps vanish as well by point (ii) above, the edge and vertex components of $\underline{v}_h$ on the interface $\Gamma$ from the side of $\Omega_\INT$ are zero.
  The same reasoning as for $\Omega_\EXT$ can therefore be applied (proceeding from the interface $\Gamma$ towards the interior of $\Omega_\INT$) to show that all elements, edge, and vertex values of $\underline{v}_h$ in $\Omega_\INT$ vanish, thus concluding the proof.
\end{proof}

\subsection{Main results}\label{sec:discrete.problem:main.results}

\begin{lemma}[Stability]\label{lem:stability}
  Assume 
  \begin{equation}\label{eq:stability.condition}
    \eta > \frac{C_{\rm tr}^2 N_\partial}{4 \epsilon}
  \end{equation}
  for some real number $0 < \epsilon < 1$.
  Then, for all $\underline{v}_h \in \underline{V}_{h,0}^k$, it holds
  \begin{equation}\label{eq:stability}
    C_{\rm stab} \| \underline{v}_h \|_{{\rm en},h}^2 \le  a_h(\underline{v}_h, \underline{v}_h),
  \end{equation}
  with $C_{\rm stab} \coloneq \min \left\{ 1 - \epsilon, \eta - \frac{C_{\rm tr}^2 N_\partial}{4 \epsilon} \right\}$.
\end{lemma}

\begin{proof}
  See Section~\ref{sec:proofs:stability}
\end{proof}

\begin{theorem}[Error estimate]\label{thm:error.estimate}
  Let $u$ be the weak solution to \eqref{eq:strong} and $\underline{u}_h$ solve \eqref{eq:discrete}.
  Under assumption \eqref{eq:stability.condition}, and further assuming that $u \in {\big[ C^0(\overline{\Omega}_\INT) \cap H^{r+2}(\Th^\INT) \big]} \times {\big[ C^0(\overline{\Omega}_\EXT) \cap H^{r+2}(\Th^\EXT) \big]}$ for some $r \in \{0,\ldots,k\}$, it holds
  \[
  \| \underline{u}_h - \Ih u \|_{{\rm en},h} \lesssim \overline{\sigma}^{\nicefrac12} h^{r+1} | u |_{H^{r+2}(\Th)},
  \]
  where $| \cdot |_{H^{r+2}(\Th)}$ is the broken $H^{r+2}$-seminorm on the mesh $\Th$ and
  the hidden constant depends only on the domain, the stability constant $C_{\rm stab}$ in \eqref{eq:stability}, the polynomial degree $k$, and the mesh regularity parameter (but is independent of both the meshsize and $\sigma$).
\end{theorem}

\begin{proof}
  See Section~\ref{sec:proofs:error.estimate}.
\end{proof}

\begin{remark}[Robustness in $\sigma$]\label{rem:sigma:robustness}
  Notice that the right-hand side of the above estimate does not depend on the ratio $\frac{\sigma_\INT}{\sigma_\EXT}$, making it robust in the case of media with highly contrasting properties.
  Crucial to obtain this robustness property is the use of weighted averages in the spirit of \cite{Burman.Zunino:06,Di-Pietro.Ern.ea:08}.
\end{remark}


\section{Numerical tests}\label{sec:numerical.tests}

To numerically assess the theoretical results of Section~\ref{sec:discrete.problem:main.results}, 
we have implemented in Python the lowest-order version of the scheme corresponding to $k = 0$.

\subsection{Square interface}\label{sec:numerical.tests:square}

We consider the square domain $\Omega = \left(-\nicefrac12, \nicefrac12\right)^2$, with a square interface
\[
\Gamma = \left\{(x,y) \in \left[-\nicefrac14,\nicefrac14\right]^2 \;:\; \text{$|x| = \nicefrac14$ or $|y| = \nicefrac14$} \right\}.
\]
Since the interface is a polygonal chain, no geometric error is introduced.
We consider the following family of solutions $u = (u_\INT, u_\EXT)$ parametrized by the ratio $\frac{\sigma_\EXT}{\sigma_\INT}$ and depicted in Figure~\ref{fig:square.test:exact.solution}:
\begin{equation}\label{eq:square.test:exact.solution}
  u_\INT = \frac{\sigma_\EXT}{\sigma_\INT} (x^2-y^2),\qquad
  u_\EXT = x^2 - y^2
\end{equation}
with forcing term, (non-homogeneous) boundary conditions, and values for $J_\Gamma$  and $\Phi_\Gamma$ inferred from the expression of $u$.

\begin{figure}
  \centering
  \includegraphics[width=\textwidth]{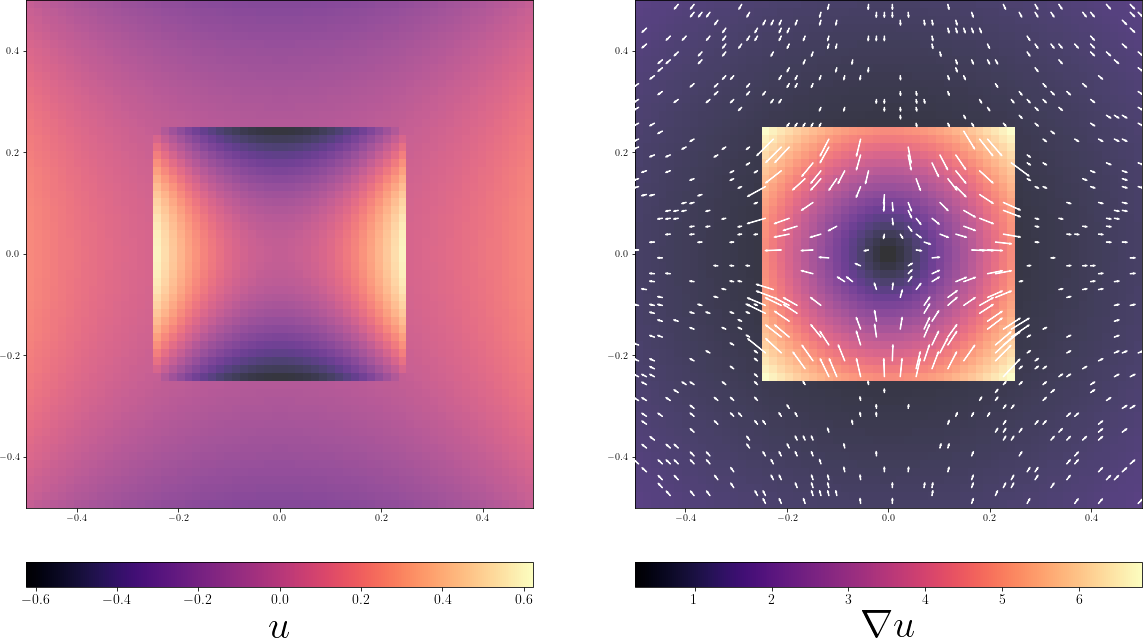} 
  \caption{The exact solution \eqref{eq:square.test:exact.solution} considered in Section~\ref{sec:numerical.tests:square} and its gradient for $\frac{\sigma_\INT}{\sigma_\EXT} = 10^{-1}$}
  \label{fig:square.test:exact.solution}
\end{figure}

We consider two mesh families, both compliant with the interface.
The first sequence is composed of Cartesian orthogonal meshes.
The second sequence is obtained from the latter by randomly moving vertices that are not located on the interface within a circle of radius $\frac{\ell}5$, with $\ell$ denoting the measure of the sides of the element in the non-deformed mesh; see Figure~\ref{fig:numerical.tests:square:meshes}.

\begin{figure}\centering
  \includegraphics[width=0.40\textwidth]{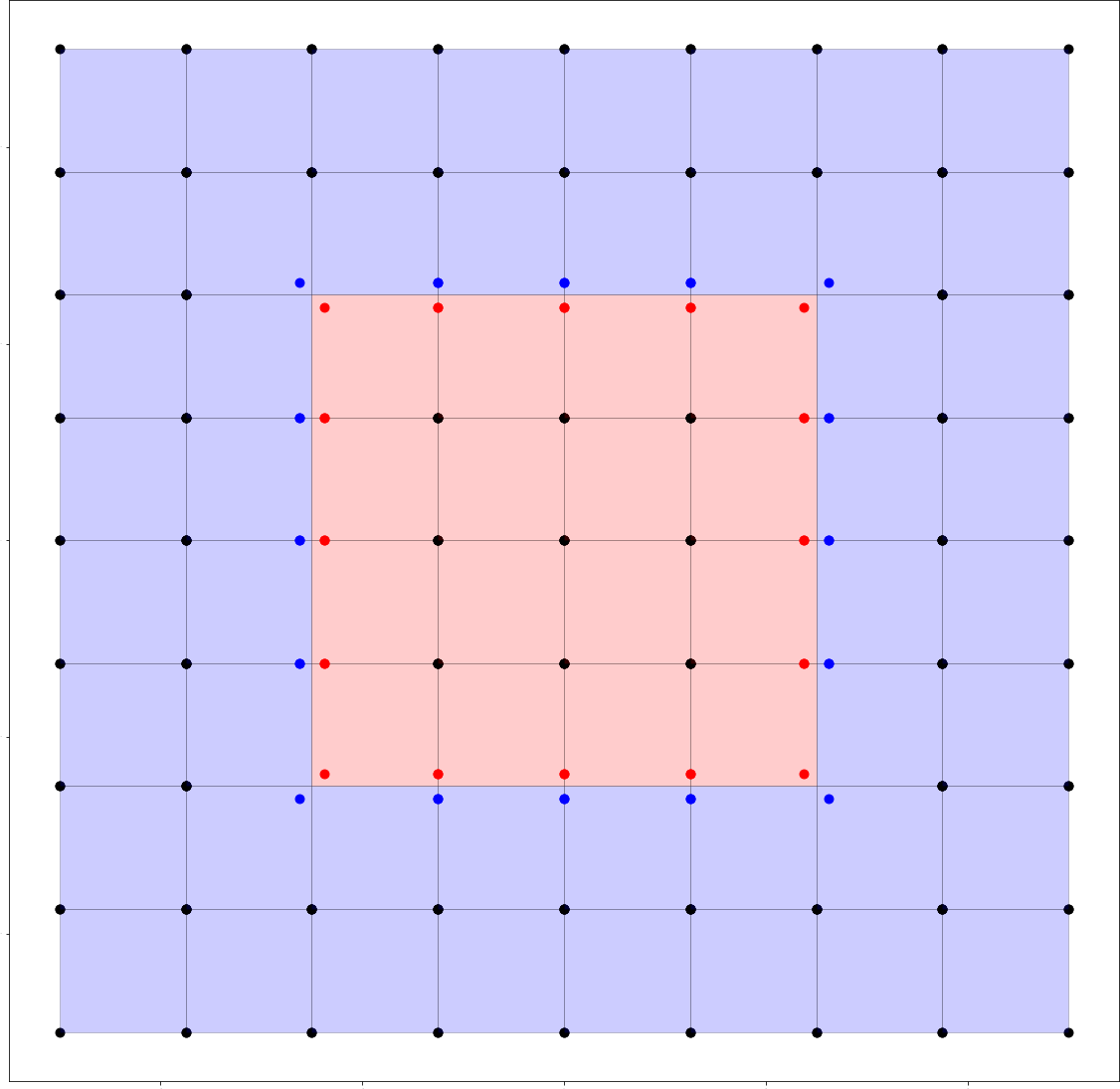}
  \includegraphics[width=0.40\textwidth]{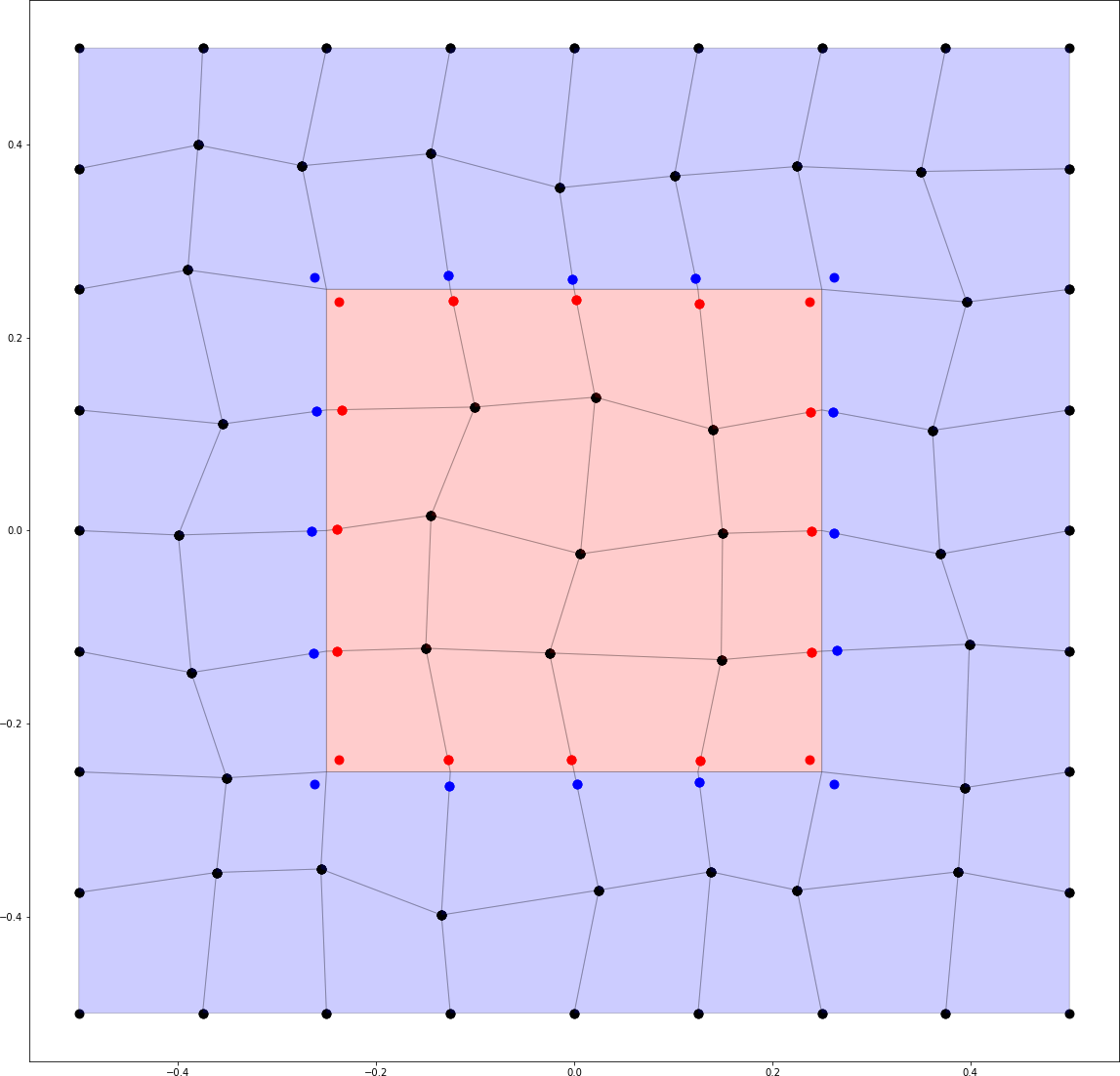}
  \caption{Mesh sequences considered in the numerical test of Section~\ref{sec:numerical.tests:square}.\label{fig:numerical.tests:square:meshes}}
\end{figure}

In order to assess the robustness of the method with respect to the ratio $\frac{\sigma_\INT}{\sigma_\EXT}$,
we let this quantity vary in $\{10^{-6}, 10^{-3}, 10^3, 10^6\}$.
We monitor two measures of the error: the energy norm defined by \eqref{eq:energy.norm} and the component $L^2$-norm $\| \cdot \|_{0,h}$ defined by \cite[Eq.~(4.20)]{Di-Pietro.Droniou:21}, i.e.,
\[
\| \underline{v}_h \|_{0,h}
\coloneq \left(
\sum_{T \in \Th} \| v_T \|_{L^2(T)}^2
+ h_T \sum_{E \in \ET}  \| v_E \|_{L^2(E)}^2
\right)^{\frac12}
\qquad \forall \underline{v}_h \in \underline{V}_h^k.
\]
In all the cases, the error is normalized with respect to the corresponding norm of the discrete solution.

\begin{figure}\centering
  \includegraphics[width=\textwidth]{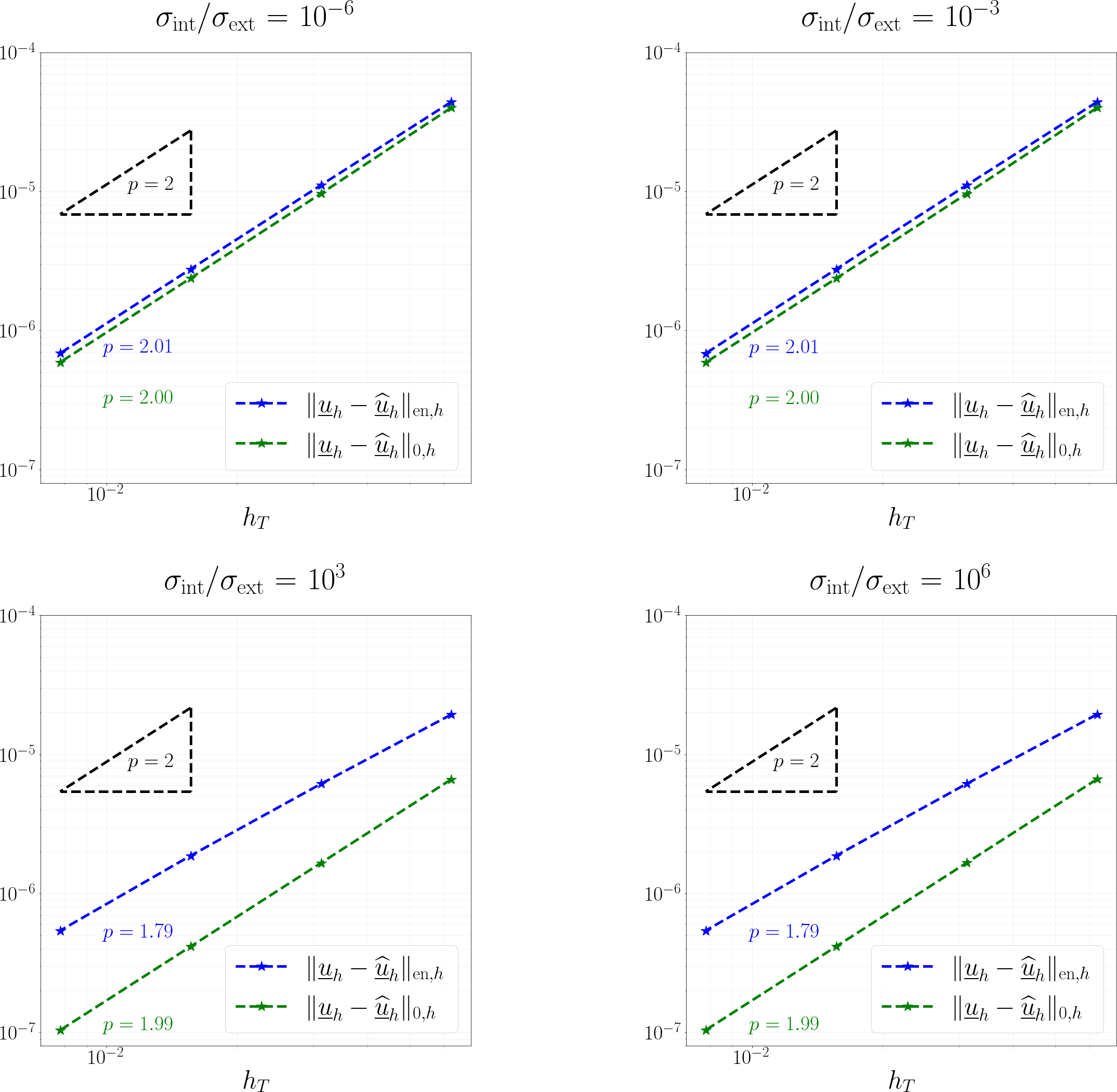}
  \caption{Convergence for different values of $\frac{\sigma_\INT}{\sigma_\EXT}$ over a mesh sequence of Cartesian orthogonal meshes,
    as described in Section~\ref{sec:numerical.tests:square}. 
    Error is normalized with respect to the norm of the reference solution.
    \label{fig:numerical.tests:square:convergence.cartesian}}
\end{figure}

\begin{figure}
  \centering
  \includegraphics[width=\textwidth]{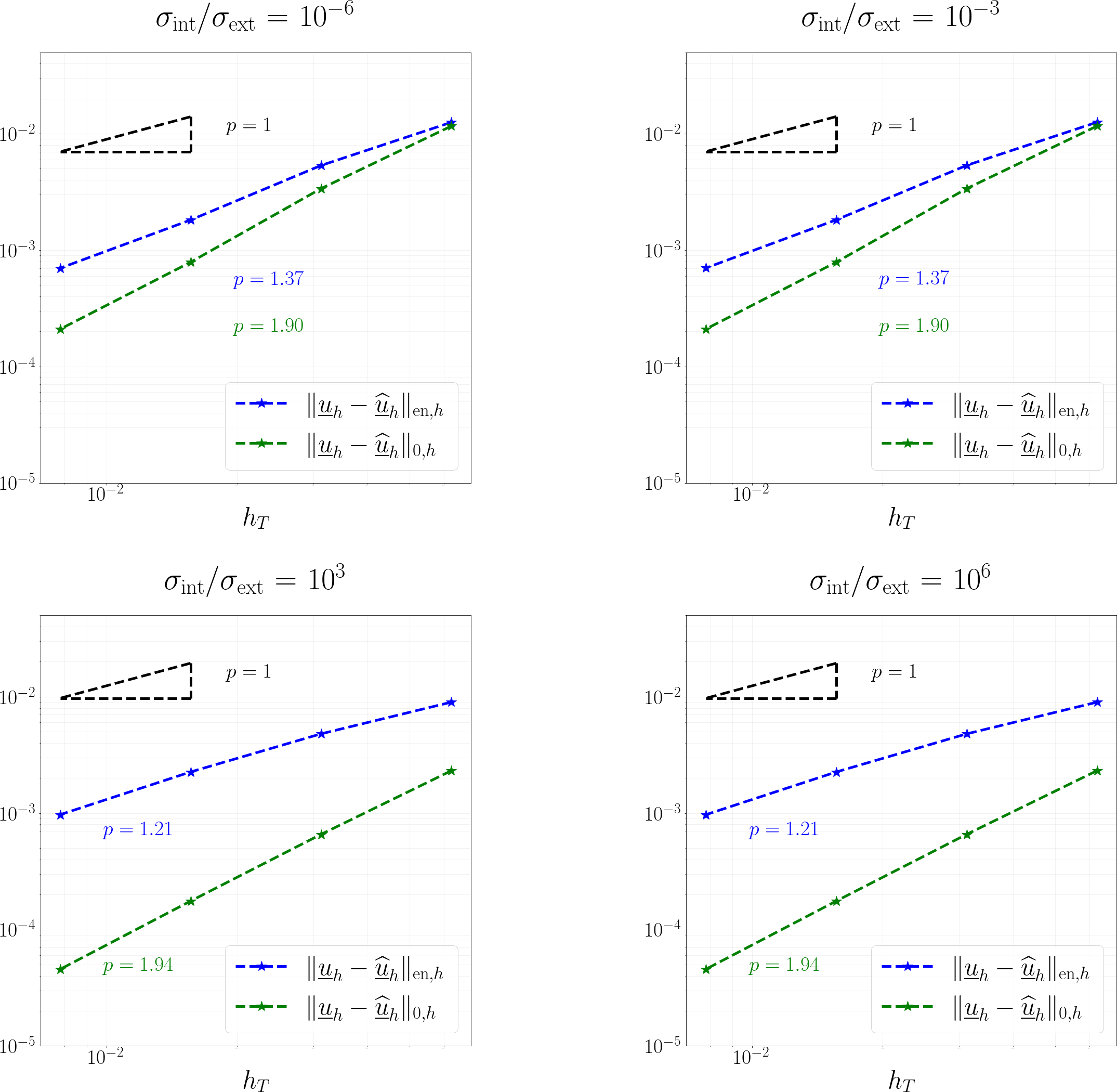}
  \caption{ Convergence for different values of $\frac{\sigma_\INT}{\sigma_\EXT}$ over a 
    mesh sequence of irregular quadrilaterals as described in Section~\ref{sec:numerical.tests:square}.
    Error is normalized with respect to the norm of the reference solution
    \label{fig:numerical.tests:square:convergence.crooked}}
\end{figure}

The results reported in Figure ~\ref{fig:numerical.tests:square:convergence.cartesian} 
and~\ref{fig:numerical.tests:square:convergence.crooked} show that the energy norm converges with order 1 (or slightly more),
as predicted by Theorem~\ref{thm:error.estimate} with $r = 0$.
We additionally notice that the error is of comparable magnitude irrespectively of the value of $\frac{\sigma_\INT}{\sigma_\EXT}$, which confirms the robustness of the method with respect to the jumps of the diffusion coefficients discussed in Remark~\ref{rem:sigma:robustness}.
As for the error in the $L^2$-like norm, convergence is close to second order, but its magnitude varies significantly with the ratio $\frac {\sigma_\INT}{\sigma_\EXT}$.
This is to be expected, since the norm $\| \cdot \|_{0,h}$ does not incorporate any dependence on the value of $\sigma$.

\subsection{Circular interface}\label{sec:numerical.tests:circle}

The second test introduces an additional difficulty, namely the fact that we deal with a curved interface.
More specifically, in the square domain  $\Omega = [-\nicefrac12, \nicefrac12]^2$, we consider the circular interface $\Gamma = \{(x,y): x^2+ y^2 = R^2 \}$ with $R=\nicefrac14$. 
The convergence of the method is tested 
considering the following family of solutions 
$u= (u_\INT, u_\EXT)$, represented in Figure ~\ref{fig:circle_intf_solution.png}:

\begin{equation}\label{eq:solution_circle}
  u_\INT = \frac{2\sigma_\EXT}{\sigma_\EXT+\sigma_\EXT}x 
  \qquad
  u_\EXT = 1 + \left[
    1 + \left(
    \frac{\sigma_\EXT-\sigma_\INT}{\sigma_\EXT+\sigma_\INT}
    \right)
    \frac{R^2}{x^2 +y^2}
    \right] x.
\end{equation}

\begin{figure}
  \centering
  \includegraphics[width=\textwidth]{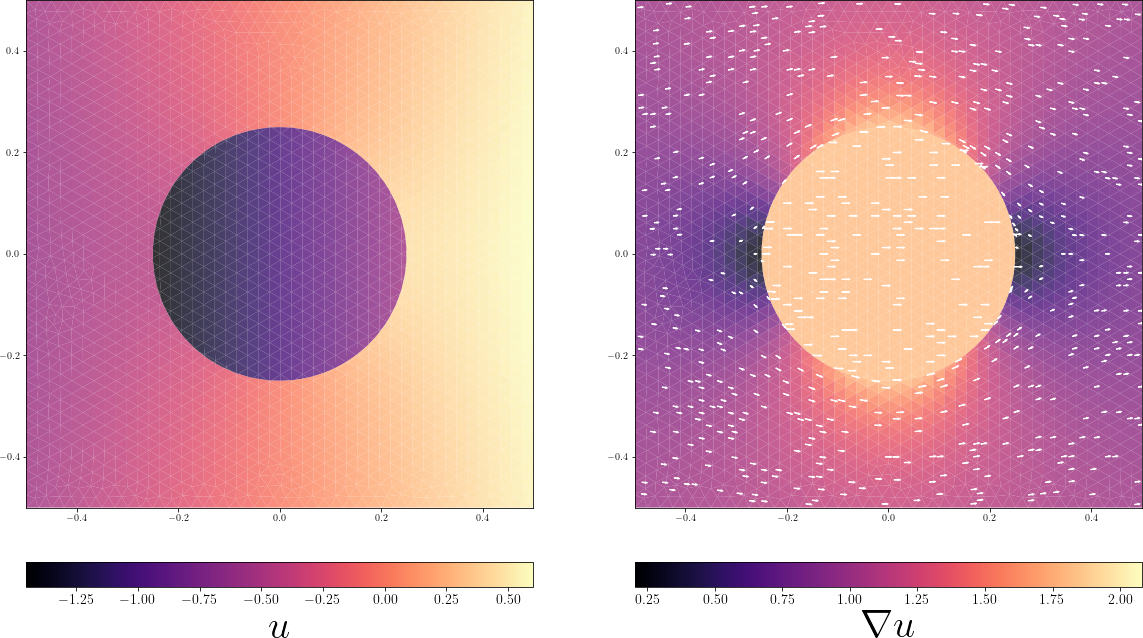} 
  \caption{ 
    The exact solution \ref{eq:solution_circle} considered in Section~\ref{sec:numerical.tests:circle} and its 
    gradient with $\frac{\sigma_\INT}{\sigma_\EXT} = 10^{-1}$.
  }
  \label{fig:circle_intf_solution.png}
\end{figure}

We consider a sequence of unstructured triangular meshes of $\Omega$ with mesh size halved at each refinement step and a family of polygonal discretizations of $\Gamma$ with segment length divided by $2^M$ at each refinement step (the integer $M$ therefore represents the refinement ratio of the interface with respect to the background mesh).
A fitted mesh is generated by splitting the elements of the original triangular mesh cut by the interface, as represented in Figure \ref{fig:circle_mesh_dofs}. 
The test is then repeated with $M=4$ for different values of $\sigma_\INT/\sigma_\EXT$ taken in  $\{10^{-6}, 10^{-3}, 10^3, 10^6\}$ to assess the convergence and robustness properties of the method; see Figure ~\ref{fig:ratio_sequence_convergence_circle}.
As for the test of Section ~\ref{sec:numerical.tests:square}, slightly more than the theoretical convergence rate $1$ is obtained for the energy norm.
In order to explore the impact of the refinement ratio, in Figure ~\ref{fig:refinement_speed_comparison} we let $\sigma_\INT/\sigma_\EXT = 10^{-1}$ and solve for several values of $M$.
The results suggest that $M=2$ is sufficient to get the theoretical convergence rate $1$, showing the ability of the method to capture curved interfaces without increasing the number of interface edges.

\begin{figure}
  \centering
  \begin{subfigure}[b]{0.45\textwidth}
    \centering
    \includegraphics[width=\textwidth]{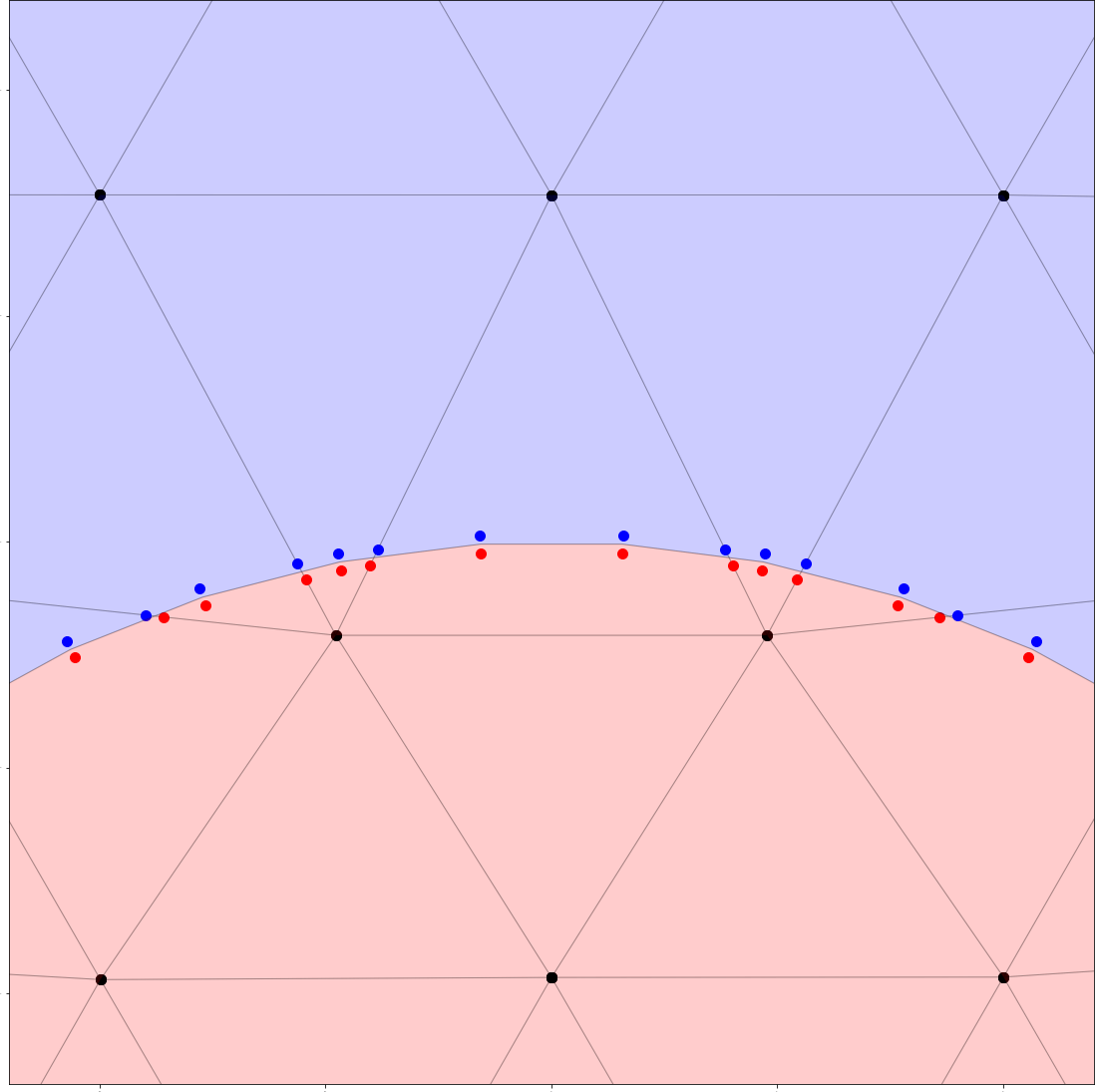}
  \end{subfigure}
  \hfill
  \begin{subfigure}[b]{0.45\textwidth}
    \centering
    \includegraphics[width=\textwidth]{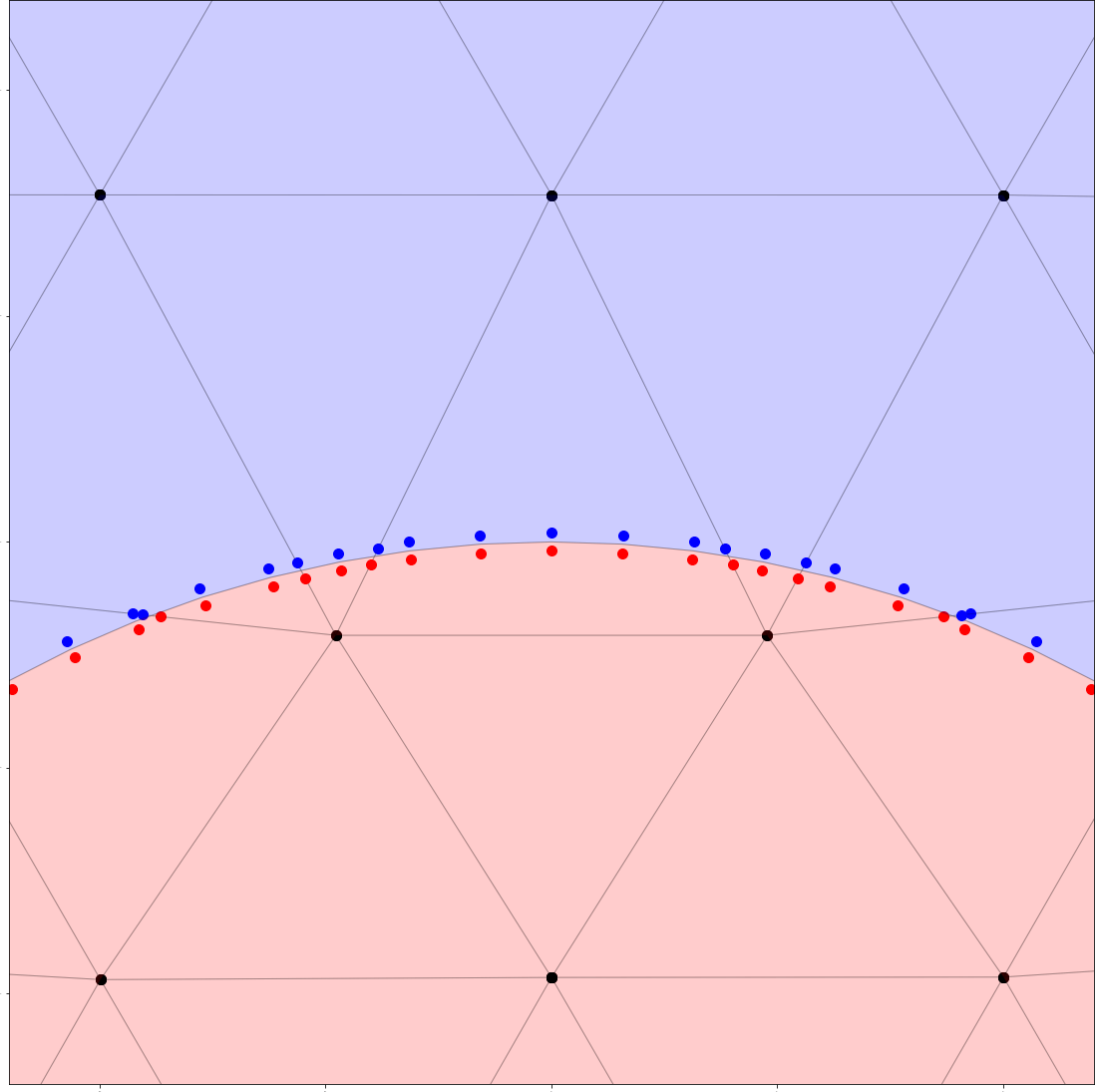}
  \end{subfigure}

  \caption{Mesh family used in the example form Section ~\ref{sec:numerical.tests:circle}.
    The detail shows new elements generated by cutting a triangular 
    mesh with a polygonal discretization of the interface. 
    Spots represent the distribution of degrees of freedom. The 
    same background triangular mesh can be cut using different refinement 
    levels for the interface.
    The accuracy of discretization of $\Gamma$ is therefore arbitrary.}
  \label{fig:circle_mesh_dofs}
\end{figure}

\begin{figure}
  \centering
  \includegraphics[width=0.45\textwidth]{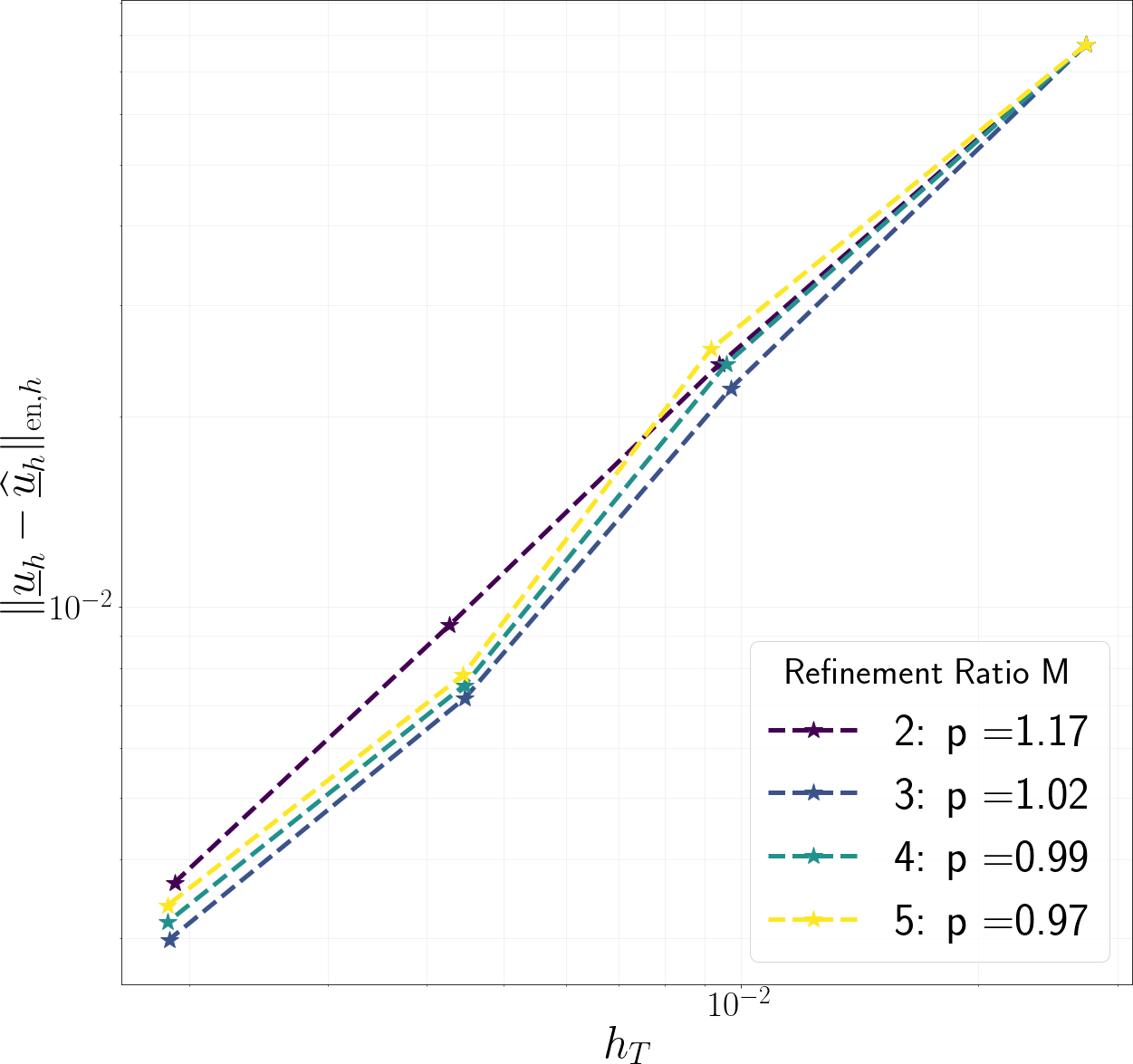}
  \caption {Convergence test described in Section~\ref{sec:numerical.tests:circle}, 
    keeping $\sigma_\INT/\sigma_\EXT = 0.1$ and varying the refinement ratio $M$.}
  \label{fig:refinement_speed_comparison}
\end{figure}

\begin{figure}
  \centering
  \includegraphics[width=\textwidth]{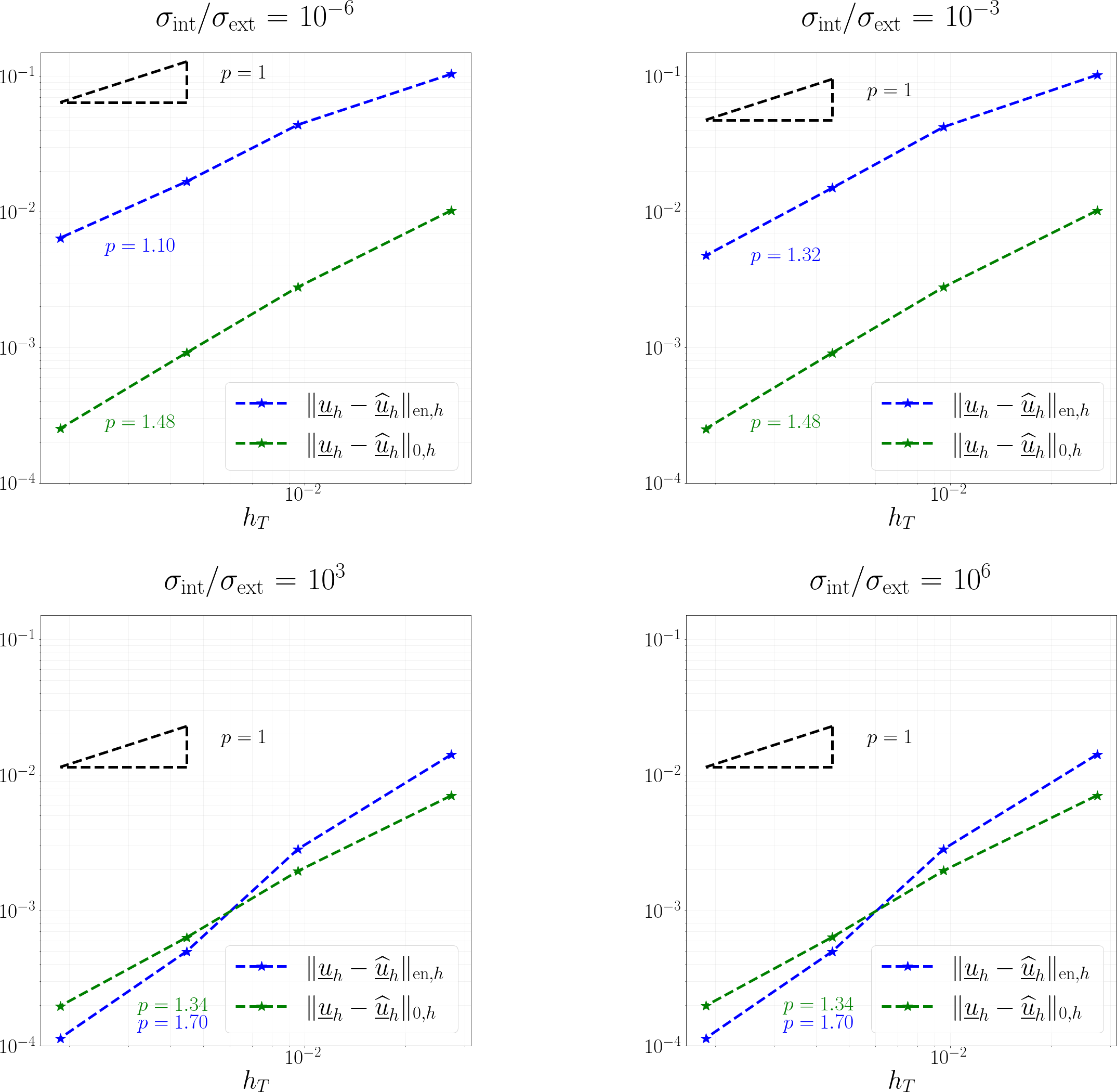}
  \caption
      {Convergence test from Section~\ref{sec:numerical.tests:circle}. 
        With refinement ratio $M=4$ the test is repeated for several 
        values of $\frac{\sigma_\INT}{\sigma_\EXT}$. Error is normalized 
        with respect to the norm of the reference solution.}
      \label{fig:ratio_sequence_convergence_circle}
\end{figure}

\subsection{Generic interface}\label{sec:numerical.tests:generic}

In the square domain  
$\Omega = [-\nicefrac12, \nicefrac12]^2$, we consider a last test where the interface is obtained by deforming a circle.
The additional difficulty comes from the fact that the curvature is no longer constant.
To test the convergence of the method, we consider the family of polynomial solution ~\ref{eq:square.test:exact.solution} used for the case of a square interface depicted in Figure ~\ref{fig:generic_intf_solution}.
Keeping the refinement ratio $M=2$, a convergence test showed in ~\ref{fig:generic_intf_convergence} is realized.
The convergence rate over 1 confirms the theoretical prediction of Theorem~\ref{thm:error.estimate}.

\begin{figure}
  \centering
  \includegraphics[width=\textwidth]{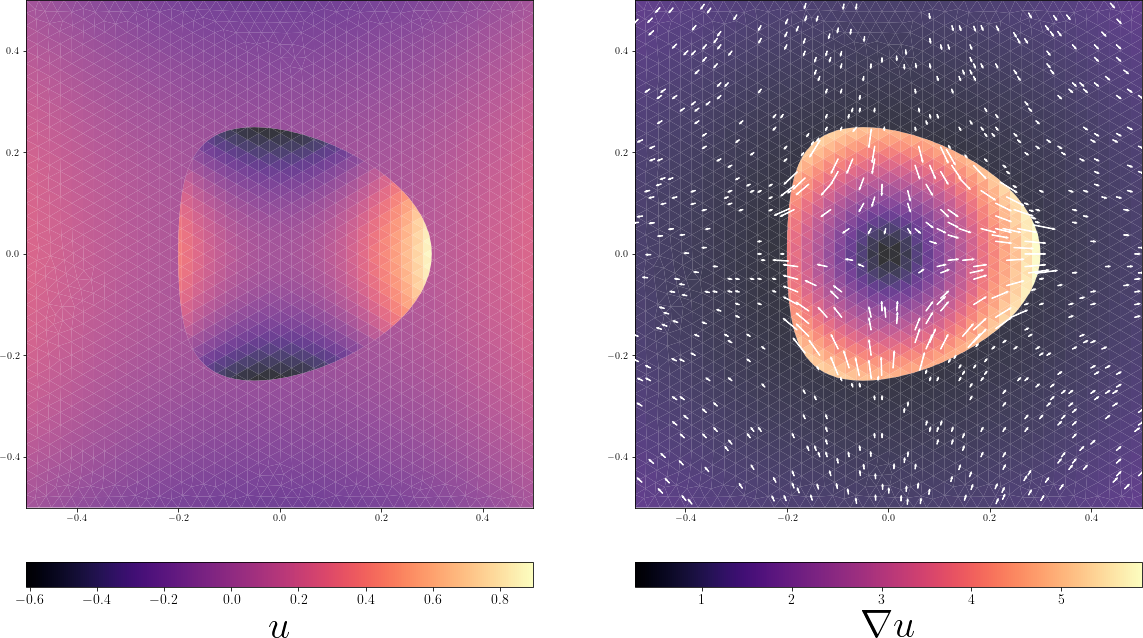} 
  \caption{The exact solution~\eqref{eq:square.test:exact.solution} considered 
    in Section~\ref{sec:numerical.tests:generic} and its gradient  
    with $\frac{\sigma_\INT}{\sigma_\EXT} = 10^{-1}$.}
  \label{fig:generic_intf_solution}
\end{figure}

\begin{figure}
  \begin{subfigure}[b]{0.45\textwidth}
    \centering
    \includegraphics[width=\textwidth]{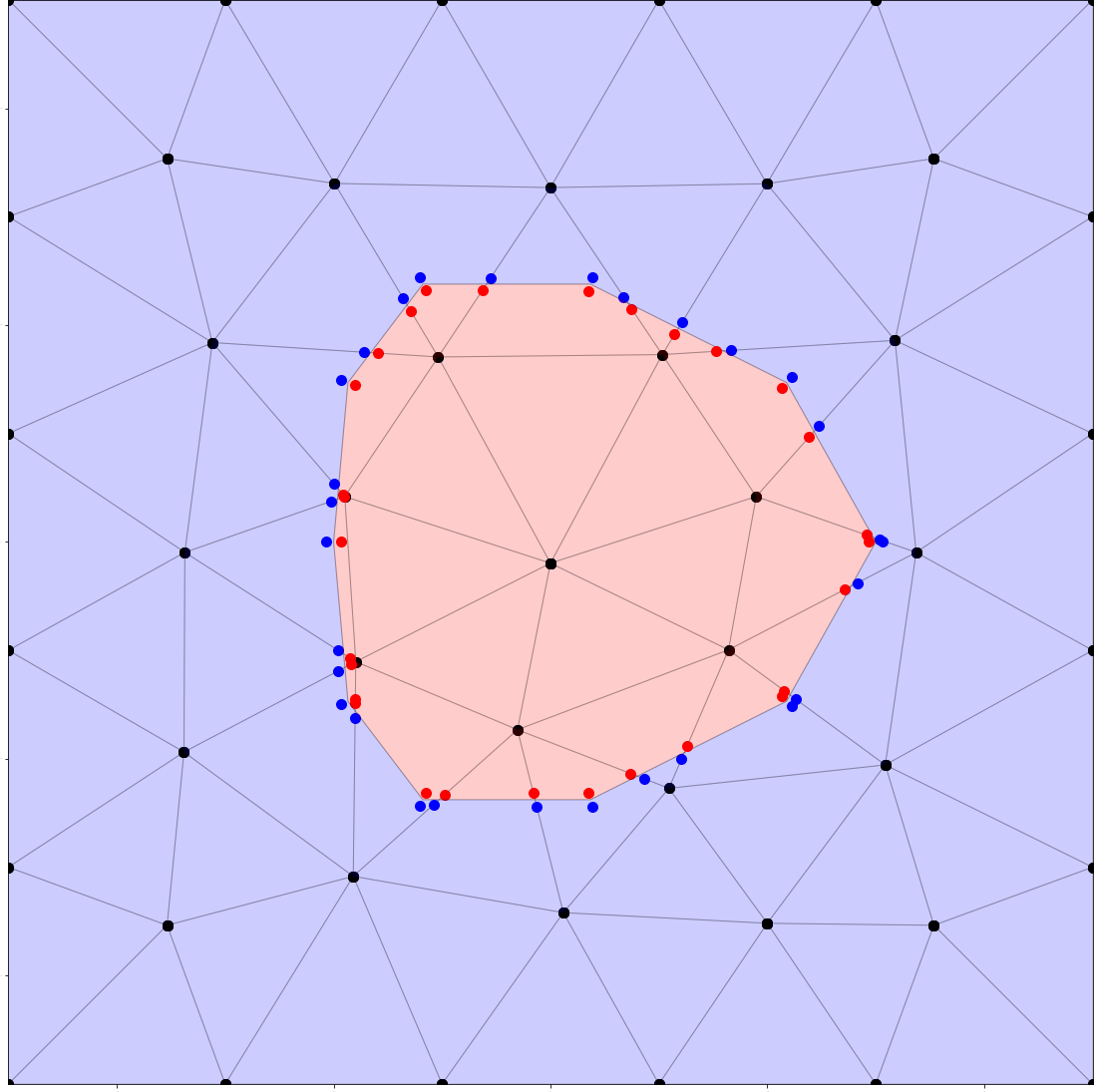}
  \end{subfigure}
  \hfill
  \begin{subfigure}[b]{0.45\textwidth}
    \centering
    \includegraphics[width=\textwidth]{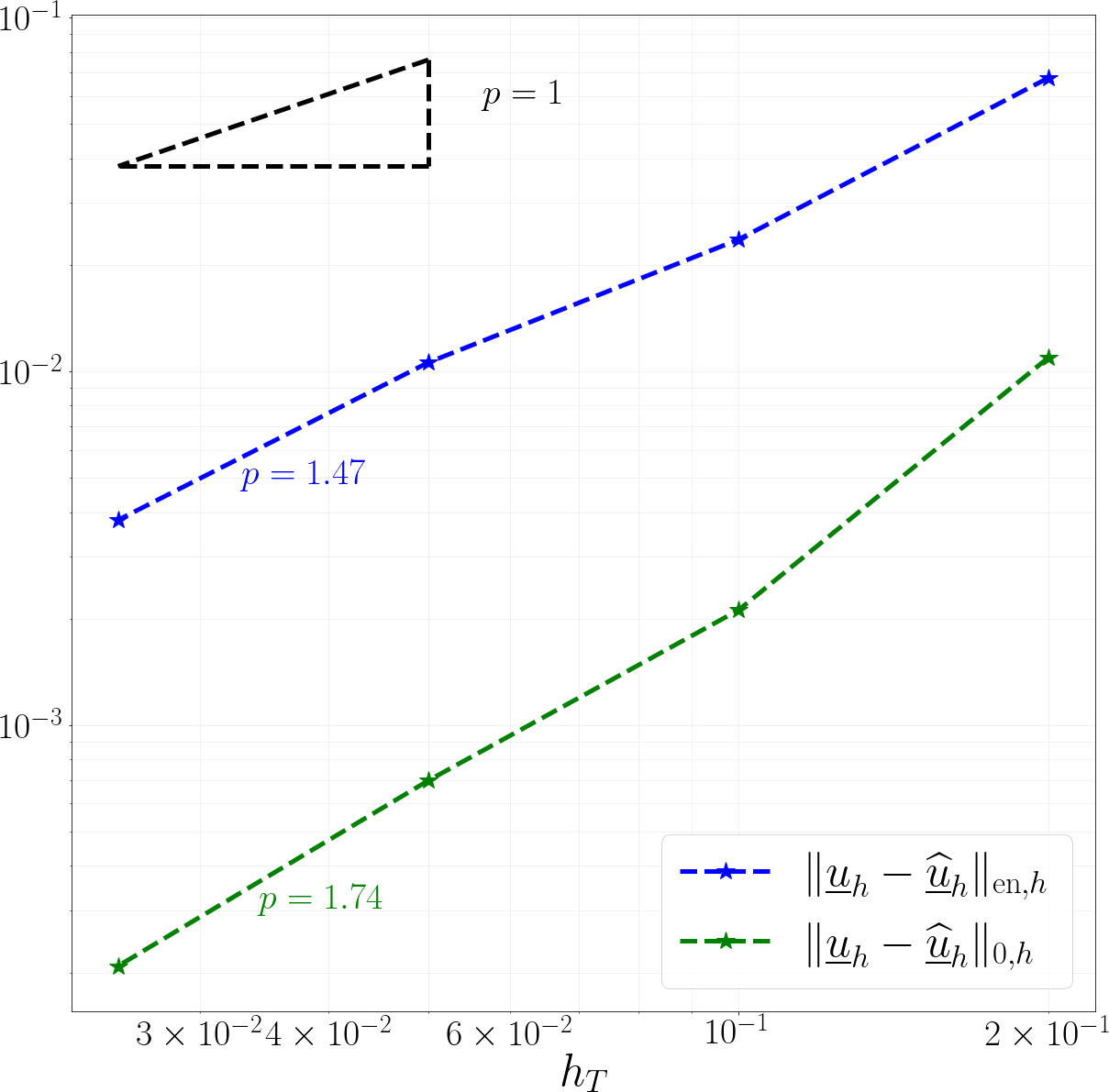}
  \end{subfigure}
  \caption{Convergence test for the case of a generic interface of Section 
    ~\ref{sec:numerical.tests:generic}. 
    On the left, a member of the mesh sequence. The red area represents $\Omega_\INT$ and 
    the blue one $\Omega_\EXT$. Dots represent the distribution of 
    degrees of freedom for the first element of the mesh sequence.
    On nodes belonging to $\Gamma$ they are doubled. On the right, $p$ 
    indicates the estimated convergence order.
    \label{fig:generic_intf_convergence}}
\end{figure}

\section{Application to the Leaky Dielectric Model}\label{sec:ldm}

In this section we discuss a version of problem \eqref{eq:strong} where the interface jump $J_\Gamma$ is time dependent and obeys an evolution equation depending on the interface gradient of $u$.

\subsection{Continuous setting}

Given a final time $\tF > 0$, a source term $f:(0,\tF\rbrack\to\Real$, and an initial potential jump $J_\Gamma^0$, we consider the problem of finding the time-dependent potential $u = (u_\INT, u_\EXT)$ with $u_\bullet : \Omega_\bullet\times (0,\tF\rbrack \to \Real$, $\bullet \in \{ \INT, \EXT \}$ and the interface jump 
$J_\Gamma: \Gamma \times [0,\tF] \to \Real$ such that
\begin{subequations}\label{eq:strong_time_dep}
  \begin{alignat}{4}
    -\nabla\cdot(\sigma_\bullet \nabla u_\bullet) &= f
    &\qquad& \text{in $\Omega_\bullet \times (0,\tF\rbrack$, $\bullet \in \{ \INT, \EXT \}$}, 
    \\
    \jumpG{u} &= J_\Gamma &\qquad& \text{on $\Gamma\times (0,\tF\rbrack$}, 
    \\
    \jumpG{\sigma\nabla u} \cdot n_\Gamma &= \Phi_\Gamma &\qquad& \text{on $\Gamma\times (0,\tF\rbrack$},       
    \\
    u_\EXT &= 0 &\qquad& \text{on $\partial\Omega\times (0,\tF\rbrack$}
    \\
    C \partial_t{J_\Gamma} &= \sigma_\bullet \nabla u_\bullet \cdot n_\Gamma &\qquad & \text{on $\Gamma\times (0,\tF\rbrack$, $\bullet \in \{ \INT, \EXT \}$}
    \label{subeq:jump_evol}
    \\ J_\Gamma(\cdot,0) &= J_\Gamma^0 &\qquad &  \text{on  $\Gamma$},
  \end{alignat}
\end{subequations}
with $C > 0$.
Problem \eqref{eq:strong_time_dep} models a situation where two media with electric conductivity respectively equal to $\sigma_\INT$ and $\sigma_\EXT$ occupy the regions $\Omega_\INT$ and $\Omega_\EXT$.
The interface $\Gamma$ between the two media is characterized by a capacitance $C$.
The variation rate of the charge in the bulk is $f$ and the interface supports a surface charge $\Phi_\Gamma$.
The region $\Omega_\bullet$ with $\bullet \in \{ \INT, \EXT \}$ is characterized by an electrostatic potential $u_\bullet$ to determine. 
The potential is discontinuous at the interface, with a jump $J_\Gamma$ to determine, and vanishes on the boundary of $\Omega$.

\subsection{Discrete problem}

To adapt the scheme \eqref{eq:discrete} to problem \eqref{eq:strong_time_dep}, it is necessary to introduce a time stepping scheme and a suitable discrete space to describe the new variable $J_\Gamma$.
For the sake of simplicity, we describe the adaption in the case of $k=0$.
Consider $N \ge 1$ time steps with duration $\tau=\nicefrac{\tF}{N}$.
For any time-dependent variable $w$, we introduce the set of time-independent variables $\{w^n\}_{n\leq N}$ such that $w^n(x) = w(x, n \tau)$.
An explicit Euler scheme is adopted to replace \eqref{subeq:jump_evol} with:
\[
\frac{C}{\tau} \big(J_\Gamma^{n+1} - J_\Gamma^n \big) 
= \sigma_\INT \nabla u^n_\INT \cdot n_\Gamma
= \sigma_\EXT \nabla u^n_\EXT \cdot n_\Gamma.
\]
The equation is integrated along $\Gamma$ after multiplying by a test function $Q_\Gamma^n:\Gamma \to \Real$:
\[
\frac{C}{\tau} 
\left(
\int_\Gamma J_\Gamma^{n+1} Q_\Gamma^n  - \int_\Gamma J_\Gamma^n Q_\Gamma^n
\right) 
= \int_\Gamma ( \sigma_\INT \nabla u^n_\INT \cdot n_\Gamma ) Q_\Gamma^n
= \int_\Gamma ( \sigma_\EXT \nabla u^n_\EXT \cdot n_\Gamma ) Q_\Gamma^n.
\]
Consider the collection of
interface vertices $\VGh$ and introduce the space
of interface variables:
\[
\underline{Z}_{\Gamma,h} \coloneq 
\Big\{ 
\underline{p}_{\Gamma,h} = (p_V)_{V\in \VGh} \st p_V \in \Real \quad
\forall  V \in \VGh
\Big\}. 
\]
Given $\underline{p}_{\Gamma,h}\in \underline{Z}_{\Gamma,h}$ and $E\in \EGh$, define $p_E\in\Poly{1}(E)$ as the unique affine function that takes the value $p_V$ at every endpoint $V$ of $E$.
Equip $\underline{Z}_{\Gamma,h}$  with the following norm:
\[
\| p_{\Gamma,h}\|_{0,\Gamma,h}^2 \coloneq \sum_{E\in \EGh} \| p_E \|_{E}^2.
\]
We set $\underline{J}_{\Gamma,h}^0 \coloneq (J_\Gamma^0(x_V))_{V \in \VGh}$, with $x_V$ denoting the coordinate vector of the vertex $V$ and, for $n = 0, \ldots, N-1$, we advance in time solving the following problem:
Find $\underline{J}^{n+1}_{\Gamma,h} \in \underline{Z}_{\Gamma,h}$ such that
\begin{equation}
  \frac{C}{\tau} 
  \bigg( \sum_{E\in \EGh} \int_E J_E^{n+1} Q_E^n -
  \int_E J_E^n Q_E^n \bigg)  
  =\sum_{E\in \EGh} \int_E
  \avg{\sigma G_T \underline{u}_{T}} \cdot n_\Gamma Q_E
  \qquad \forall \underline{Q}^n_{\Gamma,h} \in \underline{Z}_{\Gamma,h}.
  \label{eq:jump_pb}
\end{equation}
Given $\underline{J}_{\Gamma,h}^n \in \underline{Z}_{\Gamma,h}^k$, 
denote by $\mathcal{M}_h: \underline{Z}_{\Gamma,h}^k \to \underline{V}_{h}^0$ the operator
that associates to a jump $\underline{J}_\Gamma^n$ the solution $\underline{u}_h$ 
of the stationary problem \eqref{eq:discrete}. Likewise, call 
$\mathcal{N}_h: \underline{V}_{h}^0 \times \underline{Z}_{\Gamma,h} 
\to \underline{Z}_{\Gamma, h}$ the operator
that associates to a potential $\underline{u}_h$ and a jump 
$\underline{J}_{\Gamma,h}^n$ the jump $\underline{J}_{\Gamma,h}^{n+1}$
solution to problem \eqref{eq:jump_pb}. 
Then, the time advancement algorithm for the case of an evolving 
jump reads:
Given $\underline{J}^0_{\Gamma,h}$, for $n = 0, \ldots, N-1$, set, in this order,
\[
\begin{aligned}
  \underline{u}^n_h &= \mathcal{M}_h (\underline{J}_\Gamma^n), \\
  \underline{J}^{n+1}_{\Gamma, n} &= 
  \mathcal{N}_h(\underline{u}_h, \underline{J}^n_{\Gamma,h}).
\end{aligned}
\]

\subsection{Numerical tests}

To numerically assess the performance method, we consider a test case with $f=0$, $\Phi_\Gamma=0$, $J_\Gamma^0 = 0$.
This set of conditions is encountered in the description of the electric potential in the context of the Leaky Dielectric Model, and represents a situation where neither the bulk nor the surface support electric charge. 
Consider a circular interface of radius $R > 0$ immersed in a uniform far field  $E\in\Real^2$, such that $\lim_{\|x\|\to\infty} \nabla u = E$, with $\| \cdot \|$ denoting the Euclidean norm in $\Real^2$.
The solution reads:
\begin{equation}\label{eq:solution_time_dep}
  \widehat{u}(x,t) = \exp\left(-\frac{t}{t_{\rm c}}\right)(u^0(x)-u^\infty(x)) + u^\infty(x),
\end{equation}
with
\[
\begin{gathered}
  u^0 = 
  \begin{cases}
    E\frac{2\sigma_\EXT}{\sigma_\EXT+\sigma_\INT} x \quad
    & \text{if $\|x\|<R$} \\
    E\left[
      1 + \left(
      \frac{\sigma_\EXT-\sigma_\INT}{\sigma_\EXT+\sigma_\INT}
      \right)
      \frac{R^2}{\|x\|^2}
      \right] x \quad
    & \text{if $\| x \|>R$}, \end{cases}\\
  u^\infty = 
  \begin{cases}
    0
    & \text{if $\| x \|<R$}
    \\
    E \left(1+\frac{R^2}{\|x\|^2} \right) x
    & \text{if $\| x \|>R$}
  \end{cases}
\end{gathered}
\]
and $t_{\rm c} = C R \left(\frac{1}{\sigma_\INT}+\frac{1}{\sigma_\EXT}\right)$ 
(see Figure~\ref{fig:solution_time_dep}).
The system evolves from an initial condition with no potential jump at the interface to a condition of electrostatic equilibrium, where the current flow through the interface $\nabla u_\bullet \cdot n_\Gamma$ with $\bullet\in\{\INT, \EXT\}$ vanishes from either side and the electric field inside the interface is completely screened out. 

\begin{figure}
  \centering
  \includegraphics[width=0.75\textwidth]{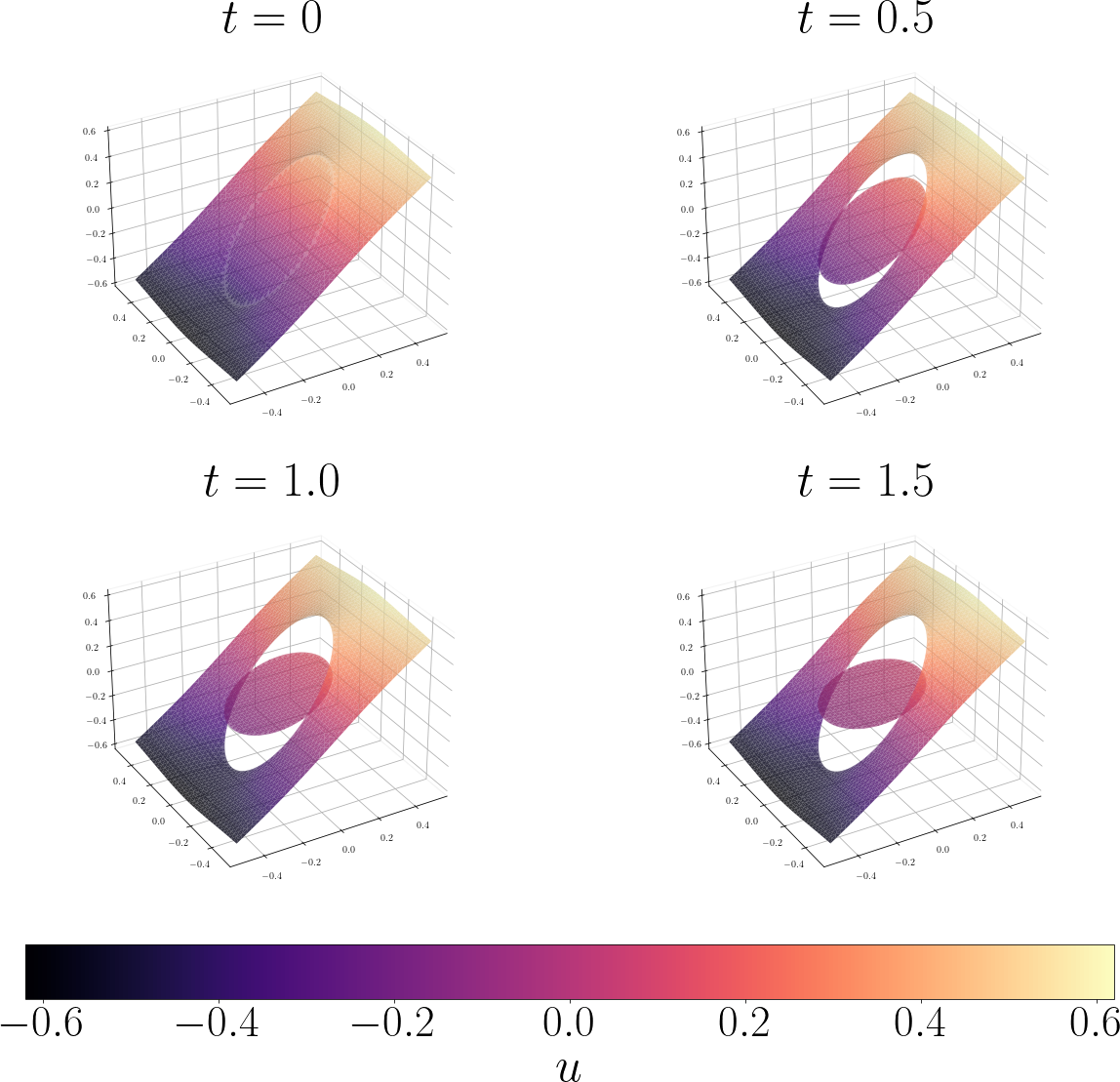}
  \caption{Solution ~\ref{eq:solution_time_dep} presented in Section~\ref{sec:ldm}. 
    Test performed with $R=\nicefrac14$, $E=1$, $\sigma_\INT/\sigma_\EXT=10^{-1}$,
    and $C$ set such that $t_{\rm c}=1$.}
  \label{fig:solution_time_dep}
\end{figure}

To monitor the convergence of the scheme, we consider a mesh sequence 
realized with the same family of triangular background meshes of 
Section ~\ref{sec:numerical.tests:circle}. The interface is refined with a 
refinement ratio $M=2$, and a sequence of time steps decreasing by a factor 4
is considered. The profile of the error for both the 
potential $u$ and the jump $J_\Gamma$ is displayed in Figure 
~\ref{fig:time_dep_convergence}. Results show that the $L^2$-temporal norm of the energy error decreases with an order slightly lower than 1.
An convergence slightly above $0.5$ is observed for the time-space $L^2$-norm of the interface jump.

\begin{figure}
  \centering
  \includegraphics[width=\textwidth]{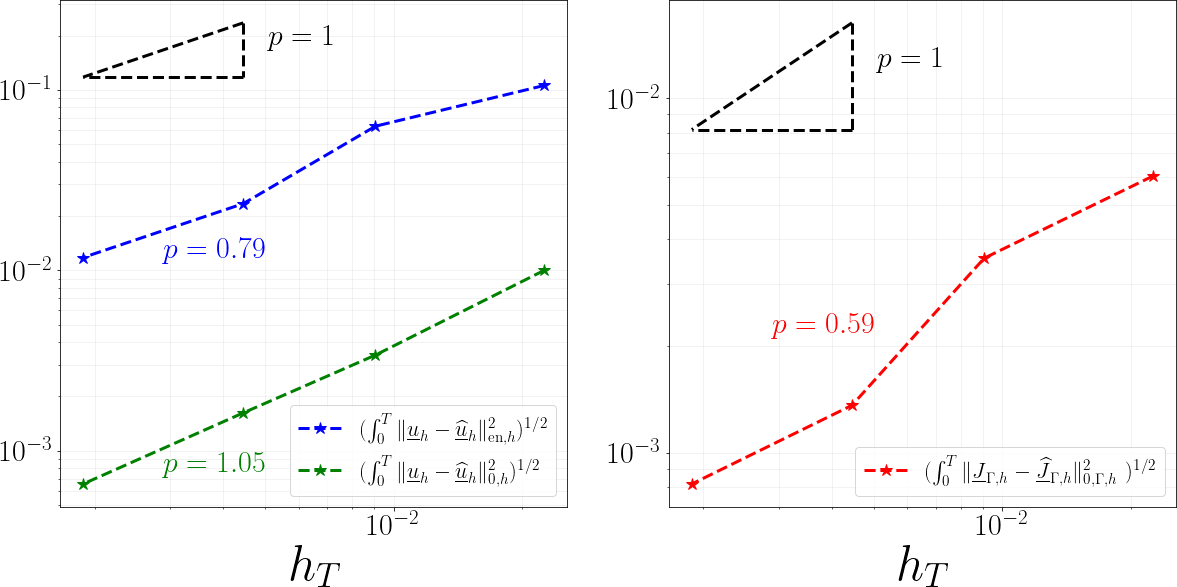}
  \caption{Convergence test for the time-dependent scheme used to reproduce 
    Solution ~\ref{eq:solution_time_dep}} presented in Section~\ref{sec:ldm}. 
  The simulation is run with $\tF=2t_{\rm c}=2$.
  \label{fig:time_dep_convergence}

\end{figure}


\section{Proofs of the main results}\label{sec:proofs}

This section collects the proofs of Lemma~\ref{lem:stability} and Theorem~\ref{thm:error.estimate}.

\subsection{Proof of Lemma~\ref{lem:stability}}\label{sec:proofs:stability}

We recall the following discrete trace inequality valid for any integer $m \ge 0$, any $T \in \Th$, and any $E \in \ET$:
For all $\varphi \in \Poly{m}(T)$,
\begin{equation}\label{eq:trace}
  \| \varphi \|_{L^2(E)} \le C_{\rm tr} h_E^{-\nicefrac12} \| \varphi \|_{L^2(T)},
\end{equation}
where $C_{\rm tr}$ only depends on $m$ and the mesh regularity parameter.

\begin{lemma}[Estimate of the consistency interface term]
  Let $N_\partial \coloneq \max_{T \in \mathcal{T}_h} \operatorname{card}(\ET \cap \EGh)$.
  Then, for all $(\underline{w}_h, \underline{v}_h) \in \underline{V}_h^k \times \underline{V}_h^k$ and any real number $\epsilon > 0$, it holds
  \begin{equation}\label{eq:consistency.term.estimate}
    \sum_{E \in \EGh} \int_E \avg{\sigma \Gh \underline{w}_h} \cdot n_E \jump{\underline{v}_h}
    \le \epsilon \| \sigma^{\nicefrac12} \Gh \underline{w}_h \|_{L^2(\Omega)^2}^2
    + \frac{C_{\rm tr}^2 N_\partial}{4 \epsilon} | \underline{v}_h |_{{\rm J},h}^2.
  \end{equation}
\end{lemma}

\begin{proof}
  Let $E \in \EGh$.
  Using a $(2,\infty,2)$-H\"{o}lder inequality, we can write
  \[
  \begin{aligned}
    &\int_E \avg{\sigma \Gh \underline{w}_h} \cdot n_E \jump{\underline{v}_h}
    \\
    &\quad
    \begin{aligned}[t]
      &\le
      \| \avg{\sigma \Gh \underline{w}_h} \|_{L^2(E)^2}
      \| n_E \|_{L^\infty(E)^2}
      \| \jump{\underline{v}_h} \|_{L^2(E)}
      \\
      \overset{\eqref{eq:avg}}&\le
      \alpha^{-\nicefrac12}\left(
      \lambda_\INT \sigma_\INT^{\nicefrac12} \| \sigma_\INT^{\nicefrac12} G_{T_\INT}^k \underline{w}_{T_\INT} \|_{L^2(E)^2}
      + \lambda_\EXT \sigma_\EXT^{\nicefrac12} \| \sigma_\EXT^{\nicefrac12} G_{T_\EXT}^k \underline{w}_{T_\EXT} \|_{L^2(E)^2}
      \right)~ \alpha^{\nicefrac12} \| \jump{\underline{v}_h} \|_{L^2(E)},
    \end{aligned}
  \end{aligned}    
  \]
  where, in the second equality, we have additionally used the fact that $\| n_E \|_{L^\infty(E)^2} \le 1$.
  Noticing that, by definition \eqref{eq:k.int.ext} of $\lambda_\bullet$ and \eqref{eq:alpha} of $\alpha$, and since $\lambda_\bullet < 1$,
  $\alpha^{-\nicefrac12} \lambda_\bullet \sigma_\bullet^{\nicefrac12} < \frac1{\sqrt{2}}$, $\bullet \in \{ \INT, \EXT \}$, we can go on writing
  \[
  \begin{aligned}
    &\int_E \avg{\sigma \Gh \underline{w}_h} \cdot n_E \jump{\underline{v}_h}
    \\
    &\quad
    \begin{aligned}[t]    
      &< \frac{1}{\sqrt{2}}\left(
      \| \sigma_\INT^{\nicefrac12} G_{T_\INT}^k \underline{w}_{T_\INT} \|_{L^2(E)^2}
      + \| \sigma_\INT^{\nicefrac12} G_{T_\EXT}^k \underline{w}_{T_\EXT} \|_{L^2(E)^2}
      \right)~ \alpha^{\nicefrac12} \| \jump{\underline{v}_h} \|_{L^2(E)}
      \\
      \overset{\eqref{eq:trace},\,\eqref{eq:seminorm.J}}&\le
      \frac{C_{\rm tr}}{\sqrt{2}}\left(
      \| \sigma_\INT^{\nicefrac12} G_{T_\INT}^k \underline{w}_{T_\INT} \|_{L^2(T_\INT)^2}
      + \| \sigma_\INT^{\nicefrac12} G_{T_\EXT}^k \underline{w}_{T_\EXT} \|_{L^2(T_\EXT)^2}
      \right)~\left(\frac{\alpha}{h_E}\right)^{\nicefrac12} \| \jump{\underline{v}_h} \|_{L^2(E)}.
    \end{aligned}
  \end{aligned}
  \]
  Summing the above inequality over $E \in \EGh$, using a Cauchy--Schwarz inequality along with the fact that $(a + b)^2 \le 2 (a^2 + b^2)$, and recalling the definition \eqref{eq:seminorm.J} of $| \cdot |_{{\rm J},h}$, we obtain
  \[
  \begin{aligned}
    &\sum_{E \in \EGh} \int_E \avg{\sigma \Gh \underline{w}_h} \cdot n_E \jump{\underline{v}_h}
    \\
    &\quad\le
    C_{\rm tr} \left[
      \sum_{E \in \EGh} \left(
      \| \sigma_\INT^{\nicefrac12} G_{T_\INT}^k \underline{w}_{T_\INT} \|_{L^2(T_\INT)^2}^2
      + \| \sigma_\INT^{\nicefrac12} G_{T_\EXT}^k \underline{w}_{T_\EXT} \|_{L^2(T_\EXT)^2}^2        
      \right)
      \right]^{\nicefrac12}
    | \underline{v}_h |_{{\rm J},h}
    \\
    &\quad\le
    C_{\rm tr} N_\partial^{\nicefrac12} \| \sigma^{\nicefrac12} \Gh \underline{w}_h \|_{L^2(\Omega)^2}
    | \underline{v}_h |_{{\rm J},h}
    \\
    &\quad\le
    \epsilon \| \sigma^{\nicefrac12} \Gh \underline{w}_h \|_{L^2(\Omega)^2}^2
    + \frac{C_{\rm tr}^2 N_\partial}{4 \epsilon} | \underline{v}_h |_{{\rm J},h}^2,
  \end{aligned}
  \]
  the conclusion being a consequence of the generalized Young's inequality.
\end{proof}

\begin{proof}[Proof of Lemma~\ref{lem:stability}]
  Recalling the expression \eqref{eq:ah} of $a_h$ and using \eqref{eq:consistency.term.estimate}, we obtain, for all $\underline{v}_h \in \underline{V}_{h,0}^k$,
  \begin{multline*}
    a_h(\underline{v}_h, \underline{v}_h)
    \ge (1-\epsilon) \| \sigma^{\nicefrac12} \Gh \underline{v}_h \|_{L^2(\Omega)^2}^2
    \\
    + \sum_{T \in \Th} \frac{\sigma_T}{h_T} \sum_{E \in \ET} \| \pT \underline{v}_T - v_{TE} \|_{L^2(E)^2}^2
    + \left( \eta - \frac{C_{\rm tr}^2 N_\partial}{4 \epsilon} \right) | \underline{v}_h |_{{\rm J},h}^2,
  \end{multline*}
  from which the conclusion readily follows recalling the definition \eqref{eq:energy.norm} of the energy norm.
\end{proof}

\subsection{Proof of Theorem~\ref{thm:error.estimate}}\label{sec:proofs:error.estimate}

\begin{lemma}[Estimate of the consistency error]\label{lem:consistency}
  Assume $u \in C^0(\overline{\Omega}_\INT) \times C^0(\overline{\Omega}_\EXT)$ and define the consistency error linear form $\mathcal{E}_h : \underline{V}_{h,0}^k \to \Real$ such that, for all $\underline{v}_h \in \underline{V}_{h,0}^k$,
  \[
  \mathcal{E}_h(\underline{v}_h)
  \coloneq \ell_h(\underline{v}_h) - a_h(\Ih u, \underline{v}_h).
  \]
  Then, provided that $u \in H^{r+2}(\Th^\INT) \times H^{r+2}(\Th^\EXT)$ for some $r \in \{0,\ldots,k\}$, it holds
  \begin{equation}\label{eq:Eh:estimate}
    \sup_{\underline{v}_h \in \underline{V}_{h,0}^k \setminus \{ \underline{0} \}} \frac{\mathcal{E}_h(\underline{v}_h)}{\| \underline{v}_h \|_{{\rm en},h}}
    \lesssim \overline{\sigma} h^{r+1} | u |_{H^{r+2}(\Th)},
  \end{equation}
  where $\overline{\sigma} \coloneq \max \{ \sigma_\INT, \sigma_\EXT \}$,
  and the hidden constant depends only on the domain, the stability constant $C_{\rm stab}$ in \eqref{eq:stability}, the polynomial degree $k$, and the mesh regularity parameter (but is independent of both the meshsize and $\sigma$).
\end{lemma}

\begin{proof}
  Let $\underline{v}_h \in \underline{V}_{h,0}^k$.
  We reformulate the components of the consistency error $\mathcal{E}_h(\underline{v}_h)$ in order to make them comparable.
  Throughout the proof we let, for the sake of brevity, $\widehat{\underline{u}}_h \coloneq \Ih u$.
  \medskip\\
  \underline{1. \emph{Reformulation of $\ell_h(\underline{v}_h)$.}}
  Recalling \eqref{eq:strong:diff_eq}, $f = -\nabla \cdot (\sigma_\bullet \nabla u_\bullet)$ almost everywhere in $\Omega_\bullet$, $\bullet \in \{ \INT, \EXT \}$.
  We can therefore write for the first term in the right-hand side of \eqref{eq:lh}:
  \begin{equation}\label{eq:lh:intermediate}
    \begin{aligned}
      \sum_{T \in \mathcal{T}_h} \int_T f \pT \underline{v}_T
      &= - \sum_{T \in \Th} \int_T \nabla \cdot (\sigma_T \nabla u) \pT \underline{v}_T
      \\
      &= \sum_{T \in \Th} \int_T \sigma_T \nabla u \cdot \nabla \pT \underline{v}_T
      - \sum_{T \in \Th} \sum_{E \in \ET} \omega_{TE} \int_E (\sigma_T \nabla u \cdot n_E) \pT \underline{v}_T
      \\
      &= \sum_{T \in \Th} \int_T \sigma_T \nabla u \cdot \nabla \pT \underline{v}_T
      - \sum_{T \in \Th} \sum_{E \in \ET} \omega_{TE} \int_E (\sigma_T \nabla u \cdot n_E) (\pT \underline{v}_T - v_{TE})
      \\
      &\quad
      - \underbrace{\sum_{T \in \Th} \sum_{E \in \ET} \omega_{TE} \int_E (\sigma_T \nabla u \cdot n_E) v_{TE}}_{\mathfrak{T}},
    \end{aligned}
  \end{equation}
  where we have used an integration by parts inside each element in the second equality and inserted $\pm v_{TE}$ into the boundary term to conclude.
  Let us focus on the last term. Rearranging the sums, we have
  \[
  \begin{aligned}
    \mathfrak{T}
    &= \sum_{E \in \Eh} \sum_{T \in \TE} \omega_{TE} \int_E (\sigma_T \nabla u \cdot n_E) v_{TE}
    \\
    &= \sum_{E \in \EGh} \int_E \left[
      (\sigma_{T_\INT} \nabla u_\INT \cdot n_E) v_{T_\INT E}
      - (\sigma_{T_\EXT} \nabla u_\EXT \cdot n_E) v_{T_\EXT E}
      \right],
  \end{aligned}
  \]
  where we have used the fact that the normal trace of $\sigma \nabla u$ is continuous across mesh edges internal to each subdomain along with the fact that $v_{TE} = 0$ on edges contained in $\partial \Omega$ in the second equality.
  We next notice that, given four real numbers $a_1$, $a_2$, $b_1$, and $b_2$, since $\lambda_\INT + \lambda_\EXT = 1$ (cf. \eqref{eq:k.int.ext},
  \[
  a_1 b_1 - a_2 b_2 = (\lambda_\INT a_1 + \lambda_\EXT a_2)(b_1 - b_2) + (a_1 - a_2) (\lambda_\EXT b_1 + \lambda_\INT b_2).
  \]
  Applying this formula with $(a_1, a_2, b_1, b_2) = (\sigma_\INT \nabla u_\INT \cdot n_E, \sigma_\EXT \nabla u_\EXT \cdot n_E, v_{T_\INT E}, v_{T_\EXT E})$,
  recalling the definitions of the interface trace operators of Section~\ref{sec:discrete.problem:trace.operators},
  and using the fact that, by \eqref{eq:strong:jump_j}, $\jump{\sigma \nabla u} \cdot n_E = \Phi_\Gamma$ almost everywhere on $\Gamma$, we infer that
  \[
  \mathfrak{T}
  = \sum_{E \in \EGh} \int_E \avg{\sigma \nabla u} \cdot n_E \jump{\underline{v}_h}
  + \sum_{E \in \EGh} \int_E \Phi_\Gamma \savg{\underline{v}_h}.
  \]
  Plugging this expression into \eqref{eq:lh:intermediate} and substituting into the expression \eqref{eq:lh} of $\ell_h$, we arrive at
  \begin{equation}\label{eq:lh:reformulation}
    \begin{aligned}
      \ell_h(\underline{v}_h)
      &= \sum_{T \in \Th} \int_T \sigma_T \nabla u \cdot \nabla \pT \underline{v}_T
      - \sum_{T \in \Th} \sum_{E \in \ET} \omega_{TE} \int_E (\sigma_T \nabla u \cdot n_E) (\pT \underline{v}_T - v_{TE})
      \\
      &\quad
      - \sum_{E \in \EGh} \int_E \avg{\sigma \nabla u} \cdot n_E \jump{\underline{v}_h}
      + \eta \sum_{E \in \EGh} \frac{\alpha}{h_E} \int_E J_\Gamma \jump{\underline{v}_h}.      
    \end{aligned}
  \end{equation}
  \underline{2. \emph{Reformulation of $a_h(\widehat{\underline{u}}_h, \underline{v}_h)$.}}
  Let $T \in \Th$.
  Writing \eqref{eq:pT.bis} for $\tau = \sigma_T \GT \widehat{\underline{u}}_T$ and rearranging, we obtain
  \[
  \int_T \sigma_T \GT \widehat{\underline{u}}_T \cdot \GT \underline{v}_T
  = \int_T \sigma_T \GT \widehat{\underline{u}}_T \cdot \nabla \pT \underline{v}_T
  - \sum_{E \in \ET} \omega_{TE} \int_E (\sigma_T \GT \widehat{\underline{u}}_T \cdot n_E) (\pT \underline{v}_T - v_{TE}).
  \]
  Substituting this expression in the definition \eqref{eq:ah} of $a_h$ written for $\underline{w}_h = \widehat{\underline{u}}_h$, we obtain
  \begin{equation}\label{eq:ah:reformulation}
    \begin{aligned}
      a_h(\widehat{\underline{u}}_h, \underline{v}_h)
      &\coloneq
      \sum_{T \in \Th} \int_T \sigma_T \GT \widehat{\underline{u}}_T \cdot \nabla \pT \underline{v}_T
      - \sum_{T \in \Th} \sum_{E \in \ET} \omega_{TE} \int_E (\sigma_T \GT \widehat{\underline{u}}_T \cdot n_E) (\pT \underline{v}_T - v_{TE})
      \\
      &\quad
      + \sum_{T \in \Th} \frac{\sigma_T}{h_T} \sum_{E \in \ET} \int_E (\pT \widehat{\underline{u}}_T - \widehat{u}_{TE}) (\pT \underline{v}_T - v_{TE})
      \\
      &\quad
      - \sum_{E \in \EGh} \int_E \avg{\sigma \Gh \widehat{\underline{u}}_h}\cdot n_E \jump{\underline{v}_h}
      + \eta \sum_{E \in \EGh} \frac{\alpha}{h_E} \int_E \jump{\widehat{\underline{u}}_h} \jump{\underline{v}_h}.
    \end{aligned}
  \end{equation}
  \underline{3. \emph{Estimate of the consistency error.}}
  Subtracting \eqref{eq:ah:reformulation} from \eqref{eq:lh:reformulation}, we arrive at the following decomposition of the consistency error:
  \begin{equation}\label{eq:Eh:decomposition}
    \mathcal{E}_h(\underline{v}_h) = \mathfrak{T}_1 + \cdots + \mathfrak{T}_5
  \end{equation}
  with
  \[
  \begin{aligned}
    \mathfrak{T}_1 &\coloneq
    \sum_{T \in \Th} \int_T \sigma_T (\nabla u - \GT \widehat{\underline{u}}_T) \cdot \nabla \pT \underline{v}_T,
    \\
    \mathfrak{T}_2 &\coloneq
    \sum_{T \in \Th} \sum_{E \in \ET} \omega_{TE} \int_E \sigma_T (\GT \widehat{\underline{u}}_T - \nabla u) \cdot n_E (\pT \underline{v}_T - v_{TE}),
    \\
    \mathfrak{T}_3 &\coloneq
    \sum_{T \in \Th} \frac{\sigma_T}{h_T} \sum_{E \in \ET} \int_E (\pT \widehat{\underline{u}}_T - \widehat{u}_{TE}) (\pT \underline{v}_T - v_{TE}),    
    \\
    \mathfrak{T}_4 &\coloneq
    \sum_{E \in \EGh} \int_E \avg{\sigma (\Gh \widehat{\underline{u}}_h - \nabla u)} \cdot n_E \jump{\underline{v}_h}    
    \\
    \mathfrak{T}_5 &\coloneq
    \eta \sum_{E \in \EGh} \frac{\alpha}{h_E} \int_E (J_\Gamma - \jump{\widehat{\underline{u}}_h}) \jump{\underline{v}_h}.
  \end{aligned}
  \]
  We next proceed to estimate the above terms.
  Using Cauchy--Schwarz inequalities along with the fact that $\sigma_T \le \overline{\sigma}$ for all $T \in \Th$, we have for the first term
  \begin{equation}\label{eq:estimate:T1}
    \begin{aligned}
      \mathfrak{T}_1
      &\le \overline{\sigma}^{\nicefrac12} \left(
      \sum_{T \in \Th} \| \nabla u - \GT \widehat{\underline{u}}_T \|_{L^2(T)^2}^2
      \right)^{\nicefrac12}
      \left(
      \sum_{T \in \Th} \sigma_T \| \nabla \pT \underline{v}_T \|_{L^2(T)^2}^2
      \right)^{\nicefrac12}
      \\
      \overset{\eqref{eq:GT:approximation},\,\eqref{eq:est.grad.pT},\,\eqref{eq:energy.norm}}&\lesssim
      \overline{\sigma}^{\nicefrac12} h^{r+1} | u |_{H^{r+2}(\Th)} \| \underline{v}_h \|_{{\rm en},h}.
    \end{aligned}
  \end{equation}
  For the second term, we use on each edge a $(2,\infty,2)$-H\"older inequality on the integral, the fact that $\| n_E \|_{L^\infty(E)^2} \le 1$ along with $\sigma_T \le \overline{\sigma}$,
  and a Cauchy--Schwarz inequality on the sums to write
  \begin{equation}\label{eq:estimate:T2}
    \begin{aligned}
      \mathfrak{T}_2
      &\le \overline{\sigma}^{\nicefrac12}
      \left(
      \sum_{T \in \Th} h_T \| \GT \widehat{\underline{u}}_T - \nabla u \|_{L^2(\partial T)^2}^2
      \right)^{\nicefrac12}
      \left(
      \sum_{T \in \Th} \frac{\sigma_T}{h_T} \sum_{E \in \ET} \| \pT \underline{v}_T - v_{TE} \|_{L^2(E)}^2
      \right)^{\nicefrac12}
      \\
      \overset{\eqref{eq:GT:approximation},\,\eqref{eq:energy.norm}}&\lesssim
      \overline{\sigma}^{\nicefrac12} h^{r+1} | u |_{H^{r+2}(\Th)} \| \underline{v}_h \|_{{\rm en},h}.
    \end{aligned}
  \end{equation}
  Cauchy--Schwarz inequalities along with $\sigma_T \le \overline{\sigma}$ yield for the third term
  \begin{equation}\label{eq:estimate:T3}
    \begin{aligned}
      \mathfrak{T}_3
      &\le \overline{\sigma}^{\nicefrac12}
      \left(
      \sum_{T \in \Th} h_T^{-1} \sum_{E \in \ET} \| \pT \widehat{\underline{u}}_T - \widehat{u}_{TE} \|_{L^2(E)}^2
      \right)^{\nicefrac12} \left(
      \sum_{T \in \Th} \frac{\sigma_T}{h_T} \sum_{E \in \ET} \| \pT \underline{v}_T - v_{TE} \|_{L^2(E)}^2
      \right)^{\nicefrac12}
      \\
      \overset{\eqref{eq:energy.norm}}&\lesssim
      \left[
        \sum_{T \in \Th} h_T^{-1} \sum_{E \in \ET} \left(
        \| \pT \widehat{\underline{u}}_T - \gamma_{TE} u \|_{L^2(E)}^2
        + \| \gamma_{TE} u - \widehat{u}_{TE} \|_{L^2(E)}^2
        \right)
        \right]^{\nicefrac12} \| \underline{v}_h \|_{{\rm en},h}
      \\
      \overset{\eqref{eq:pT:approximation},\,\eqref{eq:vTE:approximation}}&\lesssim
      \overline{\sigma}^{\nicefrac12} h^{r+1} | u |_{H^{r+2}(\Th)} \| \underline{v}_h \|_{{\rm en},h}.
    \end{aligned}
  \end{equation}
  where, in the second inequality, we have additionally used the fact that $(a + b)^2 \le 2(a^2 + b^2)$ after inserting $\pm\gamma_{TE} u$ (the trace of $u_{|T}$ on $E$) inside the norm.
  To estimate the fourth term, we start with $(2,\infty,2)$-H\"older inequalities on the integrals and Cauchy--Schwarz inequalities on the sums and recall the definition \eqref{eq:seminorm.J} to write
  \begin{equation}\label{eq:estimate:T4:basic}
    \begin{aligned}
      \mathfrak{T}_4
      \le \left(
      \sum_{E \in \EGh} \frac{h_E}{\alpha} \| \avg{\sigma (\Gh \widehat{\underline{u}}_h - \nabla u)} \|_{L^2(E)^2}^2
      \right)^{\nicefrac12} | \underline{v}_h |_{{\rm J},h}
    \end{aligned}
  \end{equation}
  Let now $E \in \EGh$ and, using the inequality $(a + b)^2 \le 2(a^2 + b^2)$, write
  \[
  \begin{aligned}
    &\alpha^{-1} \| \avg{\sigma (\Gh \widehat{\underline{u}}_h - \nabla u)} \|_{L^2(E)^2}^2
    \\
    &\quad
    \begin{aligned}[t]
      \overset{\eqref{eq:avg}}&\le
      2\alpha^{-1} \lambda_\INT \sigma_\INT \| \sigma_\INT^{\nicefrac12}  (G_{T_\INT}^k \widehat{\underline{u}}_T - \nabla u_\INT) \|_{L^2(E)^2}^2
      + 2\alpha^{-1} \lambda_\EXT \sigma_\EXT \| \sigma_\EXT^{\nicefrac12} (G_{T_\EXT}^k \widehat{\underline{u}}_T - \nabla u_\EXT) \|_{L^2(E)^2}^2
      \\
      &= \| \sigma_\INT^{\nicefrac12} (G_{T_\INT}^k \widehat{\underline{u}}_T - \nabla u_\INT) \|_{L^2(E)^2}^2
      + \| \sigma_\EXT^{\nicefrac12} (G_{T_\EXT}^k \widehat{\underline{u}}_T - \nabla u_\EXT) \|_{L^2(E)^2}^2
      \\
      \overset{\eqref{eq:GT:approximation}}&\lesssim
      \overline{\sigma}^{\nicefrac12} h^{r+1} | u |_{H^{r+2}(\TE)},
    \end{aligned}
  \end{aligned} 
  \]
  where, in the equality, we have used the fact that, by definition \eqref{eq:alpha} of $\alpha$ and \eqref{eq:k.int.ext} of $\lambda_\bullet$, $2 \alpha^{-1} \lambda_\bullet \sigma_\bullet = 1$ for $\bullet \in \{ \INT, \EXT \}$.
  Plugging the above estimate into \eqref{eq:estimate:T4:basic} and recalling the definition \eqref{eq:energy.norm} of the energy norm, we conclude that
  \begin{equation}\label{eq:estimate:T4}
    \mathfrak{T}_4 \lesssim \overline{\sigma}^{\nicefrac12} h^{r+1} | u |_{H^{r+2}(\Th)}
    \| \underline{v}_h \|_{{\rm en},h}.
  \end{equation}
  Moving to the fifth term, we recall that, by \eqref{eq:strong:jump_phi}, $J_\Gamma = \jumpG{u}$ almost everywhere on $\Gamma$ and use Cauchy--Schwarz inequalities along with the fact that $\alpha \le \frac{2 \sigma_\INT \sigma_\EXT}{2 \min \{ \sigma_\INT, \sigma_\EXT \} } \le \overline{\sigma}$ and the definition \eqref{eq:seminorm.J} of the $| \cdot |_{{\rm J},h}$-seminorm to write
  \begin{equation}\label{eq:estimate:T5}
    \begin{aligned}
      \mathfrak{T}_5 
      &\le \eta \overline{\sigma}^{\nicefrac12} \left(
      \sum_{E \in \EGh} h_E^{-1} \| \jumpG{u} - \jump{\widehat{\underline{u}}_h} \|_{L^2(E)}^2
      \right)^{\nicefrac12} | \underline{v}_h |_{{\rm J},h}
      \\
      \overset{\eqref{eq:jumpG},\,\eqref{eq:jump.savg},\,\eqref{eq:energy.norm}}&\lesssim \overline{\sigma}^{\nicefrac12} \left[
        \sum_{E \in \EGh} h_E^{-1} \left(
        \| \gamma_{T_\INT E} u - \widehat{u}_{T_\INT E} \|_{L^2(E)}^2
        + \| \gamma_{T_\EXT E} u - \widehat{u}_{T_\EXT E} \|_{L^2(E)}^2
        \right)
        \right]^{\nicefrac12} \| \underline{v}_h \|_{{\rm en},h}
      \\
      \overset{\eqref{eq:vTE:approximation}}&\lesssim
      \overline{\sigma}^{\nicefrac12} h^{r+1} | u |_{H^{r+2}(\Th)} \| \underline{v}_h \|_{{\rm en},h}
    \end{aligned}
  \end{equation}
  Plugging the estimates \eqref{eq:estimate:T1}, \eqref{eq:estimate:T2}, \eqref{eq:estimate:T3}, \eqref{eq:estimate:T4}, and \eqref{eq:estimate:T5} into \eqref{eq:Eh:decomposition}, \eqref{eq:Eh:estimate} follows.
\end{proof}

\begin{proof}[Proof of Theorem~\ref{thm:error.estimate}]
  Straightforward consequence of the Third Strang Lemma~\cite[Theorem~10]{Di-Pietro.Droniou:18} accounting for Lemmas~\ref{lem:stability} and~\ref{lem:consistency} above.
\end{proof}

\section*{Acknowledgements}

Daniele Di Pietro acknowledges the funding of the European Union (ERC Synergy, NEMESIS, project number 101115663).
Views and opinions expressed are however those of the authors only and do not necessarily reflect those of the European Union or the European Research Council Executive Agency. Neither the European Union nor the granting authority can be held responsible for them.

The authors are grateful to Prof. Matthieu Hillairet (University of Montpellier) for the precious discussions about the evolutionary potential jump 
problem described in Section~\ref{sec:ldm}.

\printbibliography

@Article{         Adjerid.Babuska.ea:23,
  Author        = {Adjerid, S. and Babu\v{s}ka, I. and Guo, R. and Lin, T.},
  Title         = {An enriched immersed finite element method for interface
                  problems with nonhomogeneous jump conditions},
  Journal       = {Comput. Meth. Appl. Mech. Engrg.},
  Volume        = {404},
  Number        = {115770},
  Year          = {2023},
  DOI           = {10.1016/j.cma.2022.115770}
}

@Article{         Antonietti.Giani.ea:13,
  Author        = {Antonietti, P. F. and Giani, S. and Houston, P.},
  Title         = {$hp$-version composite discontinuous {G}alerkin methods
                  for elliptic problems on complicated domains},
  Journal       = {SIAM J. Sci. Comput.},
  Volume        = {35},
  Number        = {3},
  Pages         = {A1417--A1439},
  Year          = {2013},
  DOI           = {10.1137/120877246}
}

@Book{            Arnold:18,
  Author        = {Arnold, D.},
  Title         = {Finite Element Exterior Calculus},
  Publisher     = {SIAM},
  Year          = {2018},
  DOI           = {10.1137/1.9781611975543}
}

@Article{         Bassi.Botti.ea:12,
  Author        = {Bassi, F. and Botti, L. and Colombo, A. and Di Pietro, D.
                  A. and Tesini, P.},
  Title         = {On the flexibility of agglomeration based physical space
                  discontinuous {Galerkin} discretizations},
  Journal       = {J. Comput. Phys.},
  Volume        = {231},
  Number        = {1},
  Pages         = {45--65},
  Year          = {2012},
  DOI           = {10.1016/j.jcp.2011.08.018}
}

@Article{         Beirao-da-Veiga.Russo.ea:19,
  Author        = {Beir\~ao da Veiga, L. and Russo, A. and Vacca, G.},
  Title         = {The {Virtual Element Method} with curved edges},
  Year          = {2019},
  Journal       = {ESAIM: Math. Model Numer. Anal.},
  Volume        = {53},
  Number        = {2},
  Pages         = {375--404}
}

@Article{         Botti.Di-Pietro:18,
  Author        = {Botti, L. and Di Pietro, D. A.},
  Title         = {Numerical assessment of {Hybrid High-Order} methods on
                  curved meshes and comparison with discontinuous {Galerkin}
                  methods},
  Journal       = {J. Comput. Phys.},
  Year          = {2018},
  Volume        = {370},
  Pages         = {58--84},
  DOI           = {10.1016/j.jcp.2018.05.017}
}

@Article{         Burman.Claus.ea:15,
  Author        = {Burman, E. and Claus, S. and Hansbo, P. and Larson, M. G.
                  and Massing, A.},
  Title         = {CutFEM: Discretizing geometry and partial differential
                  equations},
  Journal       = {Int. J. Numer. Meth. Engng},
  Volume        = {104},
  Pages         = {472--501},
  Year          = {2015},
  DOI           = {10.1002/nme.4823}
}

@Article{         Burman.Ern:18,
  Author        = {Burman, Erik and Ern, Alexandre},
  Title         = {An unfitted hybrid high-order method for elliptic
                  interface problems},
  Journal       = {SIAM J. Numer. Anal.},
  Volume        = {56},
  Year          = {2018},
  Number        = {3},
  Pages         = {1525--1546},
  DOI           = {10.1137/17M1154266}
}

@InProceedings{   Burman.Ern:19,
  Author        = {Burman, Erik and Ern, Alexandre},
  Editor        = {Radu, Florin Adrian and Kumar, Kundan and Berre, Inga and
                  Nordbotten, Jan Martin and Pop, Iuliu Sorin},
  Title         = {A Cut Cell Hybrid High-Order Method for Elliptic Problems
                  with Curved Boundaries},
  BookTitle     = {Numerical Mathematics and Advanced Applications ENUMATH
                  2017},
  Year          = {2019},
  Publisher     = {Springer International Publishing},
  Address       = {Cham},
  Pages         = {173--181},
  ISBN          = {978-3-319-96415-7}
}

@Article{         Burman.Zunino:06,
  Author        = { Burman, Erik and Zunino, Paolo},
  Title         = {A Domain Decomposition Method Based on Weighted Interior
                  Penalties for Advection‐Diffusion‐Reaction Problems},
  Journal       = {SIAM Journal on Numerical Analysis},
  Volume        = {44},
  Number        = {4},
  Pages         = {1612-1638},
  Year          = {2006},
  DOI           = {10.1137/050634736}
}

@Book{            Cangiani.Dong.ea:17,
  Author        = {Cangiani, Andrea and Dong, Zhaonan and Georgoulis,
                  Emmanuil H. and Houston, Paul},
  Title         = {{$hp$}-version discontinuous {G}alerkin methods on
                  polygonal and polyhedral meshes},
  Series        = {SpringerBriefs in Mathematics},
  Publisher     = {Springer, Cham},
  Year          = {2017},
  Pages         = {viii+131}
}

@Patent{          Coulter:53,
  Title         = {Means for counting particles suspended in a fluid},
  nationality   = {US},
  Number        = {2656508},
  Year          = {1953},
  yearfiled     = {1949},
  Author        = {Coulter, W. H.},
  day           = {20},
  dayfiled      = {27},
  monthfiled    = {Aug.}
}

@Article{         Di-Pietro.Droniou:18,
  Author        = {Di Pietro, D. A. and Droniou, J.},
  Title         = {A third {Strang} lemma for schemes in fully discrete
                  formulation},
  Year          = {2018},
  Journal       = {Calcolo},
  Volume        = {55},
  Number        = {40},
  DOI           = {10.1007/s10092-018-0282-3}
}

@Book{            Di-Pietro.Droniou:20,
  Author        = {Di Pietro, D. A. and Droniou, J.},
  Title         = {The {Hybrid High-Order} method for polytopal meshes},
  Subtitle      = {Design, analysis, and applications},
  Publisher     = {Springer International Publishing},
  Year          = {2020},
  Series        = {Modeling, Simulation and Application},
  Number        = {19},
  DOI           = {10.1007/978-3-030-37203-3}
}

@Article{         Di-Pietro.Droniou:21,
  Author        = {Di Pietro, D. A. and Droniou, J.},
  Title         = {An arbitrary-order method for magnetostatics on polyhedral
                  meshes based on a discrete {de Rham} sequence},
  Year          = {2021},
  Journal       = {J. Comput. Phys.},
  Volume        = {429},
  Number        = {109991},
  DOI           = {10.1016/j.jcp.2020.109991}
}

@Article{         Di-Pietro.Droniou:23,
  Author        = {Di Pietro, Daniele A. and Droniou, J\'er\^ome},
  Title         = {An arbitrary-order discrete {de Rham} complex on
                  polyhedral meshes: {Exactness}, {Poincar\'e} inequalities,
                  and consistency},
  Journal       = {Found. Comput. Math.},
  Volume        = {23},
  Pages         = {85--164},
  Year          = {2023},
  DOI           = {10.1007/s10208-021-09542-8}
}

@Article{         Di-Pietro.Ern.ea:08,
  Author        = {Di Pietro, D. A. and Ern, A. and Guermond, J.-L.},
  Title         = {Discontinuous {G}alerkin methods for anisotropic
                  semidefinite diffusion with advection},
  Journal       = {SIAM J. Numer. Anal.},
  Volume        = {46},
  Number        = {2},
  Pages         = {805--831},
  Year          = {2008},
  DOI           = {10.1137/060676106}
}

@Article{         Di-Pietro.Ern.ea:14,
  Author        = {Di Pietro, D. A. and Ern, A. and Lemaire, S.},
  Title         = {An arbitrary-order and compact-stencil discretization of
                  diffusion on general meshes based on local reconstruction
                  operators},
  Journal       = {Comput. Meth. Appl. Math.},
  Volume        = {14},
  Number        = {4},
  Pages         = {461--472},
  Year          = {2014},
  DOI           = {10.1515/cmam-2014-0018}
}

@Article{         Di-Pietro.Ern:15,
  Author        = {Di Pietro, D. A. and Ern, A.},
  Title         = {A hybrid high-order locking-free method for linear
                  elasticity on general meshes},
  Journal       = {Comput. Meth. Appl. Mech. Engrg.},
  Year          = {2015},
  Volume        = {283},
  Pages         = {1--21},
  DOI           = {10.1016/j.cma.2014.09.009}
}

@Article{         Droniou.Yemm:22,
  Author        = {Droniou, J. and Yemm, L.},
  Title         = {Robust hybrid high-order method on polytopal meshes with
                  small faces},
  Journal       = {Comput. Meth. Appl. Math.},
  Volume        = {22},
  Issue         = {1},
  Pages         = {47--71},
  Year          = {2022},
  DOI           = {10.1515/cmam-2021-0018}
}

@Article{         Du.Dao.ea:14,
  Title         = {Quantitative biomechanics of healthy and diseased human
                  red blood cells using dielectrophoresis in a microfluidic
                  system},
  Author        = {Du, E and Dao, M. and Suresh, S.},
  Journal       = {Extreme Mech. Lett.},
  Volume        = {1},
  Pages         = {35--41},
  Year          = {2014},
  Publisher     = {Elsevier}
}

@Article{         Honrado.Bisegna.ea:21,
  Title         = {Single-cell microfluidic impedance cytometry: From raw
                  signals to cell phenotypes using data analytics},
  Author        = {Honrado, C. and Bisegna, P. and Swami, N. S. and Caselli,
                  F.},
  Journal       = {Lab Chip},
  Volume        = {21},
  Number        = {1},
  Pages         = {22--54},
  Year          = {2021},
  Publisher     = {Royal Society of Chemistry}
}

@Article{         Huang.Wu.ea:17,
  Author        = {Huang, Peiqi and Wu, Haijun and Xiao, Yuanming},
  Title         = {An unfitted interface penalty finite element method for
                  elliptic interface problems},
  Journal       = {Comput. Methods Appl. Mech. Engrg.},
  Volume        = {323},
  Year          = {2017},
  Pages         = {439--460},
  ISSN          = {0045-7825,1879-2138},
  DOI           = {10.1016/j.cma.2017.06.004}
}

@Article{         Johansson.Larson:13,
  Author        = {Johansson, August and Larson, Mats G.},
  Title         = {A high order discontinuous {G}alerkin {N}itsche method for
                  elliptic problems with fictitious boundary},
  Journal       = {Numer. Math.},
  Volume        = {123},
  Year          = {2013},
  Number        = {4},
  Pages         = {607--628},
  ISSN          = {0029-599X,0945-3245},
  DOI           = {10.1007/s00211-012-0497-1}
}

@Article{         Mori.Young:18,
  Title         = {From electrodiffusion theory to the electrohydrodynamics
                  of leaky dielectrics through the weak electrolyte limit},
  Author        = {Mori, Yoichiro and Young, Y-N},
  Journal       = {Journal of Fluid Mechanics},
  Volume        = {855},
  Pages         = {67--130},
  Year          = {2018},
  Publisher     = {Cambridge University Press}
}

@Book{            Neumann.Sowers.ea:89,
  Title         = {Electroporation and electrofusion in cell biology},
  Author        = {Neumann, E. and Sowers, A. E. and Jordan, C. A.},
  Year          = {1989},
  Publisher     = {Springer Science \& Business Media}
}

@Article{         Saville:97,
  Title         = {Electrohydrodynamics: the Taylor-Melcher leaky dielectric
                  model},
  Author        = {Saville, D . A.},
  Journal       = {Annual review of fluid mechanics},
  Volume        = {29},
  Number        = {1},
  Pages         = {27--64},
  Year          = {1997},
  Publisher     = {Annual Reviews}
}

@Article{         Strouboulis.Babuska.ea:00,
  Author        = {Strouboulis, T. and Babu\v{s}ka, I. and Copps, K.},
  Title         = {The design and analysis of the generalized finite element
                  method},
  Journal       = {Comput. Methods Appl. Mech. Engrg.},
  Volume        = {181},
  Year          = {2000},
  Pages         = {43--69},
  DOI           = {10.1016/S0045-7825(99)00072-9}
}

@Article{         Sukumar.Moes.ea:00,
  Author        = {Sukumar, N. and Moes, N. and Moran, B. and Belytschko,
                  T.},
  Title         = {Extended finite element method for three dimensional crack
                  modelling},
  Journal       = {Int. J. Numer. Methods Engrg.},
  Volume        = {48},
  Number        = {11},
  Year          = {2000},
  Pages         = {1549--1570},
  DOI           = {10.1002/1097-0207}
}

@Article{         Sun.Morgan:10,
  Title         = {Single-cell microfluidic impedance cytometry: a review},
  Author        = {Sun, T. and Morgan, H.},
  Journal       = {Microfluid. Nanofluidics},
  Volume        = {8},
  Pages         = {423--443},
  Year          = {2010},
  Publisher     = {Springer}
}

@Article{         Taraconat.Gineys.ea:21,
  Author        = {Taraconat, P. and Gineys, J.-P. and Is\`ebe, D. and
                  Nicoud, F. and Mendez, S.},
  Journal       = {Cytom. Part A},
  Title         = {Detecting cells rotations for increasing the robustness of
                  cell sizing by impedance measurements, with or without
                  machine learning},
  Year          = {2021},
  DOI           = {10.1002/cyto.a.24356},
  Number        = {10},
  Pages         = {977-986},
  Volume        = {99}
}

@Article{         Taylor:66,
  Title         = {Studies in electrohydrodynamics. I. The circulation
                  produced in a drop by an electric field},
  Author        = {Taylor, Geoffrey Ingram},
  Journal       = {Proceedings of the Royal Society of London. Series A.
                  Mathematical and Physical Sciences},
  Volume        = {291},
  Number        = {1425},
  Pages         = {159--166},
  Year          = {1966},
  Publisher     = {The Royal Society London}
}

@Article{         Vlahovska:19,
  Title         = {Electrohydrodynamics of drops and vesicles},
  Author        = {Vlahovska, Petia M},
  Journal       = {Annual Review of Fluid Mechanics},
  Volume        = {51},
  Pages         = {305--330},
  Year          = {2019},
  Publisher     = {Annual Reviews}
}

@Article{         Yemm:24,
  Title         = {A New Approach to Handle Curved Meshes in the Hybrid
                  High-Order Method},
  Author        = {Yemm, L.},
  Journal       = {Found. Comput. Math.},
  Volume        = {24},
  Pages         = {1049--1076},
  Year          = {2024},
  DOI           = {10.1007/s10208-023-09615-w}
}
\end{document}